\documentclass[reqno, 12pt]{amsart} 

\usepackage[expansion=false]{microtype} 
\usepackage{amsfonts,amsthm,amsmath,amssymb,amscd,mathrsfs} 
\allowdisplaybreaks  
\usepackage{latexsym} 
\usepackage{xcolor, color} 
\usepackage[colorlinks=true,linkcolor=blue,citecolor=teal,urlcolor=black]{hyperref} 
\usepackage{graphicx}  
\usepackage{indentfirst} 

\renewcommand{\MR}[1]{}

\usepackage{lmodern} 
\usepackage{soul}     
\usepackage{wasysym}  

\theoremstyle{plain} 
\newtheorem{Thm}{Theorem}[section] 
\newtheorem{Lem}[Thm]{Lemma}     
\newtheorem{Prop}[Thm]{Proposition}

\theoremstyle{definition}

\theoremstyle{remark}
\newtheorem{Rem}[Thm]{Remark}
\numberwithin{equation}{section} 

\newcommand{\beq}{\begin{equation}}            
	\newcommand{\eeq}{\end{equation}}
\newcommand{\ben}{\begin{eqnarray}}         
	\newcommand{\een}{\end{eqnarray}}
\newcommand{\beno}{\begin{eqnarray*}}
	\newcommand{\eeno}{\end{eqnarray*}}

\newcommand{\lt}{\left}
\newcommand{\rt}{\right}

\newcommand{\px}{\partial_x}
\newcommand{\py}{\partial_y}

\newcommand{\qqquad}{\qquad\quad}
\newcommand{\alabel}{\stepcounter{equation}\tag{\theequation}\label}  

\usepackage[
letterpaper,                 
textheight=8.35in,           
headsep=20pt,                
footskip=36pt,               
marginparwidth=0pt,          
marginparsep=0pt,            
left=0.75in, 
right=0.75in
]{geometry}

\begin{document}
	\title[Global bounded solutions for Hillen-Painter models]{Global bounded solutions for a class of generalized Hillen-Painter models near Couette flow in $\mathbb{R}^2$}
	
	\author{Yubo~Chen}
	\address[Yubo~Chen]{School of Mathematical Sciences, Dalian University of Technology, Dalian, 116024,  China}
	\email{1220823215@mail.dlut.edu.cn}
	
	\author{Wendong~Wang}
	\address[Wendong~Wang]{School of Mathematical Sciences, Dalian University of Technology, Dalian, 116024,  China}
	\email{wendong@dlut.edu.cn}
	
	\author{Guoxu~Yang}
	\address[Guoxu~Yang]{School of Mathematical Sciences, Dalian University of Technology, Dalian, 116024,  China}
	\email{guoxu\_dlut@outlook.com}
	
   \author{Yi~Zhang}
   \address[Yi~Zhang]{School of Mathematical Sciences, Dalian University of Technology, Dalian, 116024,  China}
   \email{zysx@mail.dlut.edu.cn}
	
	\begin{abstract}
    We investigate the global well-posedness of a class of generalized Hillen--Painter systems---specifically, supercritical volume-filling chemotaxis models---in $\mathbb{R}^2$ under the influence of Couette flow. It is well established that, in the absence of fluid flow, solutions to this system may develop finite-time singularities (blow-up) for arbitrary initial cell mass. It is  proved that the introduction of a Couette flow with sufficiently large amplitude guarantees the global existence of solutions for all initial masses. By employing a novel frequency decomposition technique, we successfully remove the mass threshold limitation presented in previous studies on the domain $\mathbb{T}\times\mathbb{R}$ (Wang et al., Commun. Contemp. Math.), thereby establishing global regularity in the whole space without any smallness assumptions.

    \end{abstract}

	\maketitle
	
	{\small {\bf Keywords:} Hillen-Painter system; chemotaxis system; stability; Couette flow; blow-up suppression}

    {\small {\bf  Mathematics Subject Classification:} 35B44; 35Q92; 92C17}
	

	\section{Introduction}  
	Consider the following Patlak--Keller--Segel (PKS) system coupled with the Navier--Stokes (NS) equations in $\mathbb{R}^2$:
	\begin{equation}\begin{aligned} \label{eq:main0}
			\left\{\begin{array}{l}
				\partial_t n+v \cdot \nabla n=\nabla\cdot\lt(D(n)\nabla n\rt)-\nabla\cdot\lt(S(n)\nabla c\rt), \\
				- \Delta c=n- \alpha c , \\
				\partial_t v+v \cdot \nabla v+\nabla P=\Delta v+n \nabla \phi, \\
				\nabla \cdot v=0,
			\end{array}\right.
	\end{aligned}\end{equation}
    where $n$ denotes the microorganism density, $c$ the chemoattractant density and $v$ the velocity of the fluid. Moreover, $P$ is the scalar pressure, $\alpha \geq 0$ and $\phi$ is a given potential function accounting the effects of external forces.  For simplicity, assume that $\phi(x, y)=y$.
    Moreover, $D(n)$ represents self-diffusion and $S(n)$ denotes chemotactic sensitivity, which originate from a series of papers of Hillen-Painter (for example, see \cite{HP2001,PH2002}) by  incorporation of a population-sensing (or “quorum-sensing”) mechanism, where they assumed that
    \begin{equation}\label{eq:Qn}
        D(n)=Q(n)-nQ'(n),\quad S(n)=n Q(n),
    \end{equation}
     and $Q(n)$ states that the probability of a jump into a site decreases linearly with the cell density at that site.


    If $v=0$, $\phi=0$, $D\equiv1$ and $S\equiv id$, the system \eqref{eq:main0} is reduced to the classical Patlak--Keller--Segel system, and the origin of the PKS model is a classic narrative of interdisciplinary science derived by Patlak \cite{P1953} and Keller--Segel \cite{KS1970}. 
    For the initial mass $M:=\|n_{\rm in}\|_{L^1}<8\pi$, Blanchet--Dolbeault--Perthame \cite{BDP2006} proved the existence of bounded solutions; when $M=8\pi$, Blanchet--Carrillo--Masmoudi \cite{BCM2008} used the free energy functional to prove that infinite-time aggregation occurs; when $M>8\pi$, the solution blows up with delta Dirac formation, and the relevant proof was provided in \cite{V2002,V2004}. Wei \cite{W2018} unified and extended these results to the critical case, proving that a global mild solution exists if and only if $M\leq8\pi$. For the related progress and higher-dimensional settings, we refer the reader to \cite{N2000,CC2008,CGMN2023} and the references therein.
    
   For $v=0$, the general rates $D$ and $S$ of diffusion and cross-diffusion essentially exhibit asymptotic behavior of the form
    \begin{equation}\label{DS}
        D(n)\simeq n^{p},\quad S(n)\simeq n^{q},\quad n\simeq\infty,
    \end{equation}
    where $p\in\mathbb{R},q\in\mathbb{R}$.
    For the  system of the form (\ref{DS}) in $\mathbb{R}^N$, whether blow-up occurs depends on a critical exponent. It was established in \cite{SK2006} by Sugiyama-Kunii that the critical threshold is $p=q-\frac{2}{N}$: global weak solutions exist when $p>q-\frac{2}{N}$, whereas blow-up may occur for $p\leq q-\frac{2}{N}$ with sufficiently large initial data. The supercritical case $1<p<q-\frac{2}{N}$ is analyzed in  \cite{LS2007} by Luckhaus-Sugiyama, where the solution behaves like ``the Barenblatt solution" asymptotically as $t\rightarrow\infty$. In bounded domains, Winkler--Djie \cite{WD2010} established a similar critical threshold: the solutions are globally bounded when $p>q-\frac{2}{N}$, and finite-time blow-up occurs when $p<q-\frac{2}{N}$. In \cite{DW2024}, Ding-Winkler investigated the asymptotic behavior of the radial solution as $n\rightarrow\infty$ and showed that finite-time blow-up occurs if 
   \ben\label{eq:dingwinkler}
    \liminf_{\xi\rightarrow \infty}\frac{S(\xi)}{D(\xi)\xi^{\frac{2}{N}}}>0.
    \een
    For more references on this topic, we refer to Tao--Winkler \cite{TW2025} and the references therein.

   As shown in \cite{DCCGK2004},  chemotaxis typically occurs in fluid environments and is subject to fluid advection. We focus on whether the mixing effect of fluid flow can suppress finite-time blow-up. First, substantial progress has been made in using fluid flows to suppress chemotactic blow-up via fluid advection. Bedrossian--He \cite{BH2017} studied the suppression of blow-up by strong shear flows in $\mathbb{T}^d(d=2\;\rm or\;3)$: in the 3D case, when $M<8\pi$, a sufficiently strong shear flow can suppress solution blow-up; however, for $M>8\pi$, finite-time blow-up solutions still exist. Kiselev--Xu \cite{KX2016} provided a blow-up example and demonstrated that blow-up can be suppressed by stationary relaxation-enhancing flows and Yao--Zlatos flows in $\mathbb{T}^d(d=2\;\rm or\;3)$. He \cite{H2018} studied the suppression of blow-up by using a large strictly monotone shear flow on $\mathbb{T}\times\mathbb{R}$. Deng--Shi--Wang \cite{DSW2025} considered the blow-up suppression in $\mathbb{R}^3$ by Green's method.
    Second, for the 2D PKS system coupled with the fluid equations, Zeng--Zhang--Zi \cite{ZZZ2021} studied the PKS–NS system near Couette flow in $\mathbb{T}\times\mathbb{R}$. Li--Xiang--Xu \cite{LXX2025} suppressed the blow-up for the PKS--NS system via the Poiseuille flow in the same domain. Hu--Kiselev \cite{HK2024} suppressed blow-up in a PKS–coupled Stokes--Boussinesq system by using strong buoyancy effects. In 2D whole space, Chen--Wang--Yang \cite{CWY2025b} investigated the suppression of blow-up of solutions to the PKS(--NS) system by a large Couette flow, where they established a precise relationship between the amplitude $A$ of the Couette flow and the initial data. Moreover, for the suppression of blow-up for the 3D PKS--NS system via Couette flow, we refer to the following works: Cui--Wang--Wang \cite{CWW2025} and Cui--Wang--Wang--Wei--Yang \cite{CWWWY2025} in $\mathbb{T}\times\mathbb{I}\times\mathbb{T}$; Cui--Wang--Wang \cite{CWW2024} in $\mathbb{T}\times\mathbb{R}\times\mathbb{T}$; Cui--Wang--Wang--Wei \cite{CWWW2025} in $\mathbb{T}^3$ and the references therein.
    
   Motivated by the above progress, a natural question is whether the shear flow can suppress the blow-up of the supercritical system  \eqref{eq:main0} when $Q(n)$ behaves as in \eqref{eq:Qn}. By assuming that the chemotactic
response is switched off at high cell densities, Hillen-Painter in \cite{HP2001} showed that the response to high population
densities prevents overcrowding, and they proved local and global existence in time of
classical solutions when $Q(n)=1-n$.
   For $Q(n)=n+1$, Wang--Wang--Zhang \cite{WWZ2024} proved that a sufficiently large-amplitude shear flow can suppress blow-up of solutions in $\mathbb{T}\times\mathbb{R}$, while this result  holds for $M<\frac{2\pi}{\sqrt{3}}$. However, the sharp critical mass threshold is still unknown. Next we investigate this issue and focus on the system \eqref{eq:main0} with a general form of  $Q(n)=\kappa n+\mu$  on $\mathbb{R}^2$, and we call it \textbf{a generalized Hillen-Painter model}.
   Introduce the Couette flow, which generates horizontal advection and induces vertical diffusion. It is easy to verify that $(\tilde{n}, \tilde{c}, \tilde{v})=(0,0, U)$ with the Couette flow $U=(A y, 0)$ is a stationary solution to the system \eqref{eq:main0}. Denote the perturbation $u(t, x, y)=v(t, x, y)-U(y)$, then $(n, c, u)$ satisfies
    \begin{equation}\begin{aligned} \label{eq:main1}
    		\left\{\begin{array}{l}
    			\partial_t n+A y \partial_x n+u \cdot \nabla n-\Delta n=-\nabla \cdot \lt[n \lt(\kappa n +\mu\rt) \nabla c\rt], \\
    			-\Delta c = n - \alpha c, \\
    			\partial_t u+A y \partial_x u+\binom{A u_2}{0}+u \cdot \nabla u-\Delta u+\nabla p=\binom{0}{n}, \\
    			\nabla \cdot u=0,\\
                \left.(n, \, u)\right|_{t=0}= (n_{\mathrm{in}}, \, u_{\mathrm{in}} ),
    		\end{array}\right.
    \end{aligned}\end{equation}
    where $\kappa,\mu$ are constants. Define
    $$
    \omega = \partial_y u_1-\partial_x u_2, \quad u=\nabla^{\perp} \Phi=\left(\partial_y \Phi,-\partial_x \Phi\right),
    $$
    then
    $
    \Delta \Phi=\omega
    $
    and $\omega$ satisfies
    $$
    \partial_t \omega+A y \partial_x \omega-\Delta \omega+u \cdot \nabla \omega=-\partial_x n .
    $$
    After the time rescaling $t \mapsto A^{-1} t$, one rewrites the system \eqref{eq:main1} by
    \begin{equation}\begin{aligned} \label{eq:main}
    		\left\{\begin{array}{l}
    			\partial_t n+y \partial_x n-\frac{1}{A} \Delta n=-\frac{1}{A}(\nabla \cdot \lt[n \lt(\kappa n + \mu\rt) \nabla c\rt]   +   u \cdot \nabla n), \\
    			- \Delta c=  n - \alpha  c, \\
    			\partial_t \omega+y \partial_x \omega-\frac{1}{A} \Delta \omega=-\frac{1}{A}\left(\partial_x n+u \cdot \nabla \omega\right), \\
    			u=\nabla^{\perp} \Delta^{-1} \omega ,
                \\
                \left.(n, \, u)\right|_{t=0}= (n_{\mathrm{in}}, \, u_{\mathrm{in}} ).
    		\end{array}\right.
    \end{aligned}\end{equation}
    
    \subsection{Main result}
    To state the main theorem, we first introduce some notation. Denote $D_x := -i \partial_x$, understood as a Fourier multiplier operator. All norms involving $D_x$ are defined on the Fourier side and $\langle \cdot \rangle := \sqrt{1 + (\cdot)^2}$. For $t \geq 0$ and certain smooth enough $f$, define the anisotropic norm
    \begin{align*}
    	\|f\|_{Y_{m, \epsilon}} := \| \langle D_x\rangle^m\langle\frac{1}{D_x}\rangle^\epsilon f\|_{ L^2},
    \end{align*}
    and the space-time norm
    \begin{align*}
    	\|f\|_{X_a}^2&:= \|  e^{a A^{-\frac{1}{3}}\left|D_x\right|^{\frac{2}{3} }t}  f\|_{L^\infty L^2}^2      + \frac{1}{A} \|e^{a A^{-\frac{1}{3}}\left|D_x\right|^{\frac{2}{3} }t}  \nabla f\|_{L^2 L^2}^2   \\
    	& \quad + \frac{1}{ A^{\frac{1}{3}}}\|e^{a A^{-\frac{1}{3}}\left|D_x\right|^{\frac{2}{3} } t } \left|D_x\right|^{\frac{1}{3}} f\|_{L^2 L^2}^2  +   \|e^{a A^{-\frac{1}{3}}\left|D_x\right|^{\frac{2}{3} }t }  \partial_x \nabla  \Delta^{-1} f \|_{L^2 L^2}^2.
    \end{align*}

    \begin{Thm} \label{thm:PKSNS}
    	  For $0<\epsilon<\frac12<m$, assume that the initial data $u_{\rm in} \in H^2(\mathbb{R}^2)$, $0\leq n_{\mathrm{in}} \in L^\infty(\mathbb{R}^2) \cap L^1(\mathbb{R}^2)$, and $\omega_{\mathrm{in}} = \operatorname{curl} u_{\mathrm{in}}$ satisfying $\|\omega_{\rm in}\|_{Y_{m, \epsilon}}  + \| \langle D_x \rangle^{1/3} n_{\rm in}\|_{Y_{m, \epsilon}} <\infty.$
    	There exists a positive constant $C^*$ depending on $\epsilon, m, \kappa, \mu$ and $\alpha$ 
    	such that if  
    	$$
    	A>C^*  \left( \|n_{\mathrm{in}}\|_{L^\infty}^2\chi_{|\kappa| + |\mu| > 0}  + M^2 \chi_{|\mu| > 0} + \|\omega_{\rm in}\|_{Y_{m, \epsilon}}^2  + \| \langle D_x \rangle^\frac13 n_{\rm in}\|_{Y_{m, \epsilon}}^2    +1\right)^{3 + 12\chi_{|\kappa|>0} + 6 \chi_{|\mu|>0} - 3\chi_{|\kappa||\mu|>0}},
    	$$
    	where \(\chi\) denotes the characteristic function, the solutions to \eqref{eq:main} are global in time satisfying
    	\begin{gather*}
    		 \|  \langle D_x\rangle^m  \langle \frac{1}{D_x}\rangle^\epsilon   \omega \|_{X_a}  +  \|  \langle D_x\rangle^{m +\frac13  }  \langle \frac{1}{D_x}\rangle^\epsilon   n \|_{X_a}    
    			\leq  C\lt(\|\omega_{\rm in}\|_{Y_{m, \epsilon}}  + \| \langle D_x \rangle^\frac13 n_{\rm in}\|_{Y_{m, \epsilon}}+ 1 \rt) , \\
            \|n\|_{L^\infty L^\infty} \leq C\lt( \|n_{\mathrm{in}}\|_{L^\infty}^6 + M^6 + \|\omega_{\rm in}\|_{Y_{m, \epsilon}}^6  + \| \langle D_x \rangle^\frac13 n_{\rm in}\|_{Y_{m, \epsilon}}^6  +1\rt)
    	\end{gather*}
    	for all $t\geq0$, where $0<a<\frac{1}{16(1+2 \pi)}$ and $C$  depends on $\epsilon, m, \kappa, \mu$ and $\alpha$. Additionally, $\epsilon>\frac13$ is required when $\alpha= 0$.
    \end{Thm}

     \begin{Rem} It is worth noting that the supercritical volume-filling chemotaxis system ($\kappa>0$) can exhibit blow-up for arbitrarily initial mass in the absence of fluid flow (as established in Theorem 1.2 \cite{WYZ2026}). Here we investigated a generalized Hillen--Painter type model coupled with the Navier--Stokes equations under the influence of a strong Couette flow. For $\kappa=\mu=\alpha =1$, Wang--Wang--Zhang \cite{WWZ2024} had proven that a sufficiently large Couette flow can suppress blow-up of \eqref{eq:main} in $\mathbb{T}\times\mathbb{R}$, where  $M < \frac{2\pi}{\sqrt{3}}$ is necessary. In fact, the zero mode $n_0$ defined by
		$P_0n=n_0=\frac{1}{|\mathbb{T}|}\int_{\mathbb{T}}n(t,x,y)dx$ satisfying
        \beno
        \dfrac{d}{dt}\| n_0\|_{L^2}^2+\dfrac{2}{A}\|\partial_yn_0\|_{L^2}^2=\dfrac{2}{3A}\int_{\mathbb{R}}n_0^4dy+\cdots
        \eeno
        with the Gagliardo-Nirenberg inequality
		$$
			\|n_{0}\|_{L^{4}}\leq C_{*}\|n_{0}\|_{L^{1}}^{\frac12}\|\partial_{y}n_{0}\|_{L^{2}}^{\frac12}$$
     implies that $\|n_{0}\|_{L^{1}}$ must be small (for full details, see Lemma 3.6 in \cite{WWZ2024}), which seems necessary.
        In contrast, our result in the whole space $\mathbb{R}^2$ removes this restriction on the initial mass $M$. It seems surprising that the critical mass threshold may depend on the different domains.
    \end{Rem}

     \begin{Rem} \label{rem:domain}
        To handle horizontal convolutions in nonlinear terms, we adopt the frequency decomposition strategy from \cite{LLZ2025}. Specifically, we partition the integration domain into three regions (for example, see  \eqref{eq:Fourier}) based on the interactions among $k$, $l$, and $k-l$:
\begin{itemize}
    \item \textbf{Near-resonant region} ($\frac{1}{2}|k-l| \leq |k| \leq 2|k-l|$): here $|k| \sim |k-l|$, and $|l|$ represents the lowest frequency;
    \item \textbf{High-low interaction} ($|k| > 2|k-l|$): corresponds to $|k| \sim |l|$, with $|k-l|$ being the lowest frequency;
    \item \textbf{Low-high interaction} ($2|k| < |k-l|$): corresponds to $|k-l| \sim |l|$, with $|k|$ being the lowest frequency.
\end{itemize}
This decomposition facilitates precise estimates by exploiting the relative scales of the frequencies in each region. The essential point is that, in the whole-space setting, the horizontal
low-frequency singularity is compensated by the weighted factor
$\langle D_x^{-1}\rangle^\epsilon$ and by a trichotomy of horizontal
frequency interactions. Consequently, one can avoid the separate
zero-mode analysis which, in the domain $\mathbb T\times\mathbb R$,
leads to a mass restriction.
    \end{Rem}
    
    \begin{Rem} \label{rem:Y space}
        Define the space $Y_{\epsilon} := \big\{ f : \| \langle D_x^{-1} \rangle^\epsilon f \|_{L_x^2} < \infty \big\}$ with a small parameter $\epsilon > 0$. Compared to $L_x^2$, the space $Y_{\epsilon}$ requires additional regularity in the $x$-direction, specifically imposing a stricter constraint on the low-frequency Fourier modes. Assume the asymptotic behavior
\begin{align*}
    |\hat{f}(k)| \sim |k|^s \quad \text{as} \quad |k| \to 0,
\end{align*}
inclusion in $Y_{\epsilon}$ requires $s > \epsilon - \frac12$. This condition is strictly stronger than the standard $L^2$-integrability threshold $s > -\frac12$.
    \end{Rem}

    \begin{Rem}
        In Theorem \ref{thm:PKSNS}, $\epsilon$ needs an additional lower bound ${1}/{3}$ when $\alpha =0$, which allows us to gain additional integrability
		\begin{equation*}
			\||k|^{-\frac{5}{6}}\langle k\rangle^{-m}\langle\frac{1}{k}\rangle^{-\epsilon}\|_{L_k^2} < \infty
		\end{equation*}
		(for example, see (3.7) in \cite{CWY2025b} ) and thereby formally control $ \| \nabla^2 c \|_{L^2}$ with the help of  $ \| n\|_{L^2}$.
    \end{Rem}

    \begin{Rem}
		As seen in Theorem \ref{thm:PKSNS}, the initial $n$ requires $1/3$ more regularity in $x$ compared to $\omega$. This is a consequence of the derivative transfer $|D_x|^{1/3}$ in \eqref{eq:px n omega}, analogous to the handling of the buoyancy term in \cite{CWY2025a}.
	\end{Rem}


    \subsection{Some notations and outlines}
    
    Here are some notations used in this paper.
    
    \noindent\textbf{Notations}:
    \begin{itemize}
    	\item For a given function $f(x,y)$ on $\mathbb{R}^2$, its $k$-th horizontal Fourier modes can be defined by
    	\begin{equation} \label{eq:def of fk}
    		f_k(y)=\mathcal{F}_{x \rightarrow k}(f)(k, y) = (2\pi)^{-\frac12}\int_{\mathbb{R}} f(x, y) \mathrm{e}^{- ikx} dx,
    	\end{equation}
    	In addition, denote 
    	$$
    	\hat{f}_k(\xi)=\hat{f}(k, \xi) = \mathcal{F}_{y \rightarrow \xi}(f_k)(y) = (2\pi)^{-1}  \int_{\mathbb{R}^2} f(x, y) \mathrm{e}^{- i(kx + \xi y)} dxdy.
    	$$
        \item Denote $\nabla_l :=(l,\, \py)$ and $\Delta_l := \py^2 - l^2 $.
    	\item The space norm $\|f\|_{L^{p}}$ is defined by	
    	$\|f\|_{L^{p}(\mathbb{R}^2)}=\left(\int_{\mathbb{R}^2}|f|^p dxdy\right)^{{1}/{p}}  $. For $t\geq0$, the time-space norm $\|f\|_{L^{q}L^{p}}$ is defined by	
    	$\|f\|_{L^qL^p}=\|\|f\|_{L^p(\mathbb{R}^2)}\|_{L^q(0,t)}$.
    	For simplicity, we write $\|f\|_{L^p(\mathbb{R}^2)}$ as $\|f\|_{L^p}$ and write $\| (f,g)\|_{L^p} = \sqrt{\|f\|_{L^p}^2 + \|g\|_{L^p}^2 }$. 
    	\item Denote $(f | g)$ by the $L^2\left(\mathbb{R}^2\right)$ inner product of $f$ and $g$.
        \item Denote $C$ by  a positive constant independent of $A$, $t$ and the initial data, and it may be different from line to line. $A \lesssim B$ means there exists an absolute constant $C$, such that $A \leq C B$.
    \end{itemize}
    
    \noindent\textbf{Outlines}:
    The paper is organized as follows. In Section \ref{Sec.2}, we introduce the constructed multipliers, outline the main proof strategy and present Proposition \ref{main prop}, which is important for the proof of Theorem \ref{thm:PKSNS}. Combining this proposition with the local well-posedness of the system, we complete the proof of Theorem \ref{thm:PKSNS}. In Section \ref{sec.3}, we establish a series of key lemmas to close the energy estimates, thereby providing the proof of Proposition \ref{main prop}.    
    

    \section{Preliminaries and proofs of the main theorem} \label{Sec.2}
    
    \subsection{Construction of the multipliers}
    For $k,\xi \in \mathbb{R}$ and $A >0$, denote that 
    \begin{equation*}
    	\begin{aligned}
    		\mathcal{M}_1(k, \xi) & := \arctan \left(A^{-\frac{1}{3}}|k|^{-\frac{1}{3}} \operatorname{sgn}(k) \xi\right)+ \frac{\pi}{2},\\
    		\mathcal{M}_2(k, \xi) & := \arctan \left(\frac{\xi}{k}\right)+\frac{\pi}{2}
    		.
    	\end{aligned}
    \end{equation*}
    Then, these two self-adjoint Fourier multipliers satisfy
    \begin{equation} \label{eq:bound of M}
    	1\leq \mathcal{M}:= \mathcal{M}_1+\mathcal{M}_2
    	+1 \leq 1+ 2\pi.
    \end{equation}

    The enhanced dissipation multiplier $\mathcal{M}_1$ and the inviscid damping multiplier $\mathcal{M}_2$ are inspired by \cite{BGM2017,DWZ2021,WZ2023}. On the Fourier side, these multipliers act as `ghost weights'---bounded weights that provide additional dissipation properties. A crucial feature of these multipliers is the following identity:
    \begin{equation} \label{eq:crucial M}
    	2 \Re\big(\left(\partial_t+y \partial_x\right) f \mid \mathcal{M}_i f\big)=\frac{d}{d t}\|\sqrt{\mathcal{M}_i} f\|_{L^2}^2+\int_{\mathbb{R}^2}  k \partial_{\xi}  \mathcal{M}_i(k, \xi)|\hat{f}|^2 d k d \xi,
    \end{equation}
    which holds for any sufficiently smooth function $f= f(t, x, y)$ on $(0, +\infty) \times \mathbb{R}^2$.
    \begin{Lem} \label{lem:M12}
    	For any smooth enough function $f= f(t, x ,y)$ on $(0, +\infty) \times \mathbb{R}^2$ and Fourier multiplier $\mathcal{M}$ defined in \eqref{eq:bound of M}, it holds that
    	\begin{equation}  \label{eq:m}
    		\int_{\mathbb{R}^2}  k \partial_{\xi}   \mathcal{M}(k, \xi)|\hat{f}|^2 d k d \xi 
    		\geq \frac{1}{4 A^{\frac{1}{3}}}\|\left|D_x\right|^{\frac{1}{3}} f\|_{L^2}^2-\frac{1}{2A}\|\partial_y f\|_{L^2}^2 + \|\partial_x \nabla \Delta^{-1} f\|_{L^2}^2 
    		.
    	\end{equation}
    \end{Lem}
    \begin{proof}
    	Direct calculation shows that
    	$$
    	k \partial_{\xi} \mathcal{M}_1(k, \xi)=\frac{A^{-\frac{1}{3}}|k|^{\frac{2}{3}}}{1+A^{-\frac{2}{3}}|k|^{-\frac{2}{3}}|\xi|^2} \geq \frac{1}{4 A^{\frac{1}{3}}  }  |k|^{\frac{2}{3}}-\frac{1}{2A}|\xi|^2
    	$$
    	and
    	$$
    	k \partial_{\xi} \mathcal{M}_2(k, \xi)=\frac{k^2}{k^2+\xi^2},
    	$$
    	which imply \eqref{eq:m}.
    \end{proof}

    \subsection{A key proposition and some useful lemmas}
    Let $u$ be determined by $\eqref{eq:main}_4$, and $f$ satisfy
    \begin{equation}\begin{aligned} \label{eq:f}
    		\left\{\begin{array}{l}
    			\partial_t f+y \partial_x f-\frac{1}{A} \Delta f=-\frac{1}{A}u\cdot\nabla f+g,\quad (t,x,y) \in [0, T]\times \mathbb{R}^2, \\
    			\left. f\right|_{t=0} = f_{\rm in}, \quad (x,y) \in  \mathbb{R}^2,
    		\end{array}\right.
    \end{aligned}\end{equation}
    where $g$ is an additional term. The space-time estimate for \eqref{eq:f} is stated as follows.
    \begin{Prop}[Proposition 2.3 in \cite{CWY2025b}] \label{lem:est of f}
    	Let $u$ be determined by $\eqref{eq:main}_4$ and $0<a<\tfrac{1}{16(1+2 \pi)}$.
    	Then, for $0<\epsilon<1/2<m$ and $0\leq t\leq T$, it holds that 
    	\begin{equation*}\begin{aligned}
    			\|  \langle D_x\rangle^m  \langle \frac{1}{D_x}\rangle^\epsilon   f  \|_{X_a}^2 
    			&\lesssim    
    			\|  \langle D_x\rangle^m  \langle \frac{1}{D_x}\rangle^\epsilon   f_{\mathrm{in}} \|_{L^2}^2 +   \int_{0}^{t}\left| \Re\left(g \left\lvert\, \mathcal{M} e^{2 a A^{-\frac{1}{3}}\left|D_x\right|^{\frac{2}{3} }t }\langle D_x\rangle^{2 m}\langle\frac{1}{D_x}\rangle^{2 \epsilon} f\right.\right)\right| dt  \\
    			&\quad + \frac1{A^\frac12}\|  \langle D_x\rangle^m  \langle \frac{1}{D_x}\rangle^\epsilon   \omega \|_{X_a}   \|  \langle D_x\rangle^m  \langle \frac{1}{D_x}\rangle^\epsilon   f \|_{X_a}^2 .
    	\end{aligned}\end{equation*}
    \end{Prop}
    Then, we give a direct estimate for nonlinear terms of the form $|D_x|^{1/3}u \cdot \nabla f$.
    \begin{Lem}[Lemma A.3 in \cite{CWY2025b}] \label{lem:est of dx u}
    	Let $u$ be determined by $\eqref{eq:main}_4$ and $0<a<\tfrac{1}{16(1+2 \pi)}$. For $0<\epsilon<1/2<m$, $0\leq t\leq T$ and certain smooth enough $f$, it holds that
    	\begin{equation*}\begin{aligned}    
    			&\quad\frac1A \int_{0}^{t} \left|\Re\left(|D_x|^{\frac13}u \cdot \nabla f \left\lvert\, \mathcal{M} e^{2 a A^{-\frac{1}{3}}\left|D_x\right|^{\frac{2}{3} }t }\langle D_x\rangle^{2 m   }\langle\frac{1}{D_x}\rangle^{2 \epsilon} |D_x|^\frac13  f  \right.\right)\right|   dt\\
    			&\lesssim \frac1{A^\frac12}\| \langle D_x\rangle^m\langle\frac{1}{D_x}\rangle^\epsilon \omega \|_{X_a}    \| \langle D_x\rangle^m\langle\frac{1}{D_x}\rangle^\epsilon |D_x|^\frac13 f \|_{X_a}^2   .
    	\end{aligned}\end{equation*}
    \end{Lem}
    A fundamental elliptic estimate for $c$ is stated as follows.
    \begin{Lem}[Lemma 2.2 in \cite{CWY2025b}] \label{lem:Elliptic estimate}
    	Let $\mathcal{N} = \mathcal{N}(D_x, D_y, t)$ be a Fourier multiplier. Then, for $c$ and $n$ satisfying $\eqref{eq:main}_2$ and all $t \geq 0$, the following estimates hold.
    	
    	\textbf{Case of $\alpha = 0$:}
    	\begin{equation*}
    		\| \nabla^2 \mathcal{N} c \|_{L^2} \leq \| \mathcal{N} n \|_{L^2}, \quad
    		\| \nabla c \|_{L^4} \lesssim   \| n \|_{L^2} + M  ;
    	\end{equation*}
    	
    	\textbf{Case of $\alpha > 0$:}
    	\begin{equation*}
    		\| \nabla^2 \mathcal{N}c \|_{L^2}^2 +  	2 \alpha \| \nabla \mathcal{N}c \|_{L^2}^2 + 	\|  \alpha \mathcal{N}c \|_{L^2}^2 = 	\|  \mathcal{N}n \|_{L^2}^2, \quad 
    		\| \nabla c \|_{L^4} \leq C(\alpha) \| n \|_{L^2}.
    	\end{equation*}
    \end{Lem}
    At the end of this subsection, we give the following $L^\infty$ embedding, which is frequently used in this paper.
    \begin{Lem}[Lemma A.2 in \cite{CWY2025b}]
    	For $f_l= f_l(y) \in H^1(\mathbb{R})$ (defined in \eqref{eq:def of fk}), it holds that
    	\begin{align}
    		\| f_l(\cdot)\|_{L^\infty} &\leq    \|   f_l(\cdot)\|_{L^2}^\frac12  \|   \py f_l(\cdot)\|_{L^2}^\frac12, \label{eq:GN1}\\
    		\| f_l(\cdot)\|_{L^\infty} &\leq  |l|^{-\frac12} \| \nabla_l f_l(\cdot)\|_{L^2}. \label{eq:GN2}
    	\end{align}
    \end{Lem}

    \subsection{Bootstrap argument} 
    
    In this paper, we use the standard bootstrap argument to prove the main theorem. Let us define $T$ to be the end-point of the largest interval $[0, T]$ such that the following hypotheses hold for all $0 \leq t \leq T$ :
    \begin{align*}
    		&E(t):=  \sqrt{\|  \langle D_x\rangle^m  \langle \frac{1}{D_x}\rangle^\epsilon   \omega \|_{X_a} ^2    +   \|  \langle D_x\rangle^{m + \frac13}  \langle \frac{1}{D_x}\rangle^\epsilon    n \|_{X_a}^2 }   \leq 2K, \\
    		&E_\infty (t) :=  \| n\|_{L^\infty L^\infty}   \leq 2K_\infty, \alabel{eq:bootstap}
    \end{align*}
    where $K\geq 1$ and $K_\infty$ will be determined in the proof. We shall show the following proposition to improve the hypotheses above.
    \begin{Prop} \label{main prop}
    	Under the same assumptions of Theorem \ref{thm:PKSNS}, there exists a positive constant $C^*$ depending on $\epsilon, m, \kappa, \mu$ and $\alpha$ 
    	such that if $A> \bar{A}$ with 
    	$$
    	  \bar{A}
          = 
          C^*  \left( \|n_{\mathrm{in}}\|_{L^\infty}^2\chi_{|\kappa| + |\mu| > 0}  + M^2 \chi_{|\mu| > 0} + \|\omega_{\rm in}\|_{Y_{m, \epsilon}}^2  + \| \langle D_x \rangle^\frac13 n_{\rm in}\|_{Y_{m, \epsilon}}^2    +1\right)^{3 + 12\chi_{|\kappa|>0} + 6 \chi_{|\mu|>0} - 3\chi_{|\kappa||\mu|>0}},
    	$$
    	there hold
    	$$
    	E(t) \leq K \quad \text{and} \quad E_\infty(t) \leq K_\infty
    	$$
    	for all $0<t<T$. Additionally, $\epsilon>\frac13$ is required when $\alpha= 0$.
    \end{Prop}
    \begin{proof}[Proof of Theorem \ref{thm:PKSNS}]
    	Proposition \ref{main prop} with the local well-posedness (for example, the similar arguments as \cite{W2018}, \cite{WWZ2026} or \cite{WYZ2026}) of the system \eqref{eq:main} implies that $T=+\infty$, and thus completes the proof.
    \end{proof}

    \section{Proof of Proposition \ref{main prop}}	 \label{sec.3}
    In this section, we will estimate $n$, $\omega$ and then complete the proof of Proposition \ref{main prop}.
    
    \subsection{Estimate of $n^2$.}
    As a preliminary step, we give a lemma about the estimate of $n^2$ by a new frequency decomposition, which plays an important role in  the energy estimate of $n$.
    \begin{Lem} \label{lem:a preliminary step}
        For $a, \gamma, t\geq 0$ and $0<\epsilon< \frac12 <m$, it holds that
    	\begin{align*}
    		&\quad\big\|e^{a A^{-\frac{1}{3}}|k-l|^{\frac{2}{3} } t }\langle k-l\rangle^{m + \gamma } \langle\frac{1}{k-l}\rangle^\epsilon    \int_{\mathbb{R}}\| n_{k-l-\eta}\|_{L_y^2} \|n_{\eta}\|_{L_y^\infty} d\eta \big\|_{L_{k-l}^2} \\
    		&\lesssim  \big\|  e^{a A^{-\frac{1}{3}}|D_x|^{\frac{2}{3} } t } \langle D_x\rangle^{m + \gamma}\langle\frac{1}{D_x}\rangle^\epsilon   n    \big\|_{L^2 }
    		\big\|  e^{a A^{-\frac{1}{3}}|D_x|^{\frac{2}{3} } t } \langle D_x\rangle^{m + \gamma}\langle\frac{1}{D_x}\rangle^\epsilon |D_x|^\frac13  n    \big\|_{L^2 }^\frac12 \\
    		&\quad\times \big\|  e^{a A^{-\frac{1}{3}}|D_x|^{\frac{2}{3} } t } \langle D_x\rangle^{m + \gamma}\langle\frac{1}{D_x}\rangle^\epsilon  \py n   \big\|_{L^2 }^\frac12 . 
    	\end{align*}
    \end{Lem}
    \begin{proof}
    	We divide the frequency space into three regions as follows:
    	\begin{align*}
    		&\quad\big\|e^{a A^{-\frac{1}{3}}|k-l|^{\frac{2}{3} } t }\langle k-l\rangle^{m + \gamma  } \langle\frac{1}{k-l}\rangle^\epsilon    \int_{\mathbb{R}}\| n_{k-l-\eta}\|_{L_y^2} \|n_{\eta}\|_{L_y^\infty} d\eta \big\|_{L_{k-l}^2} \\
    		& = \big\| e^{a A^{-\frac{1}{3}}|k-l|^{\frac{2}{3} } t }\langle k-l\rangle^{m + \gamma  } \langle\frac{1}{k-l}\rangle^\epsilon    \| n_{k-l-\eta}\|_{L_y^2} \|n_{\eta}\|_{L_y^\infty}  \big\|_{L_{k-l}^2 L_\eta^1}\\
    		&\leq \big\|  e^{a A^{-\frac{1}{3}}|k-l|^{\frac{2}{3} } t }\langle k-l\rangle^{m + \gamma  } \langle\frac{1}{k-l}\rangle^\epsilon    \| n_{k-l-\eta}\|_{L_y^2} \|n_{\eta}\|_{L_y^\infty}   \big\|_{L_{k-l}^2 L_\eta^1({|k-l-\eta|}/{2} \leq |k-l| \leq 2|k-l-\eta|)} \\
    		&\quad +  \big\| e^{a A^{-\frac{1}{3}}|k-l|^{\frac{2}{3} } t }\langle k-l\rangle^{m + \gamma  } \langle\frac{1}{k-l}\rangle^\epsilon    \| n_{k-l-\eta}\|_{L_y^2} \|n_{\eta}\|_{L_y^\infty}  \big\|_{L_{k-l}^2 L_\eta^1(  2|k-l-\eta|<|k-l|  )}\\
    		&\quad +\big\| e^{a A^{-\frac{1}{3}}|k-l|^{\frac{2}{3} } t }\langle k-l\rangle^{m + \gamma  } \langle\frac{1}{k-l}\rangle^\epsilon    \| n_{k-l-\eta}\|_{L_y^2} \|n_{\eta}\|_{L_y^\infty} \big\|_{L_{k-l}^2 L_\eta^1(  2|k-l| < |k-l-\eta| )}\\
    		&=: I_1 + I_2 +I_3.  \alabel{eq:n^2}
    	\end{align*}
     The integration domain is decomposed following the same strategy as in Remark \ref{rem:domain}---namely, by treating 
$k-l$ as a single variable. For $I_1$, using \eqref{eq:GN1} and 
    	\begin{align*}
    		\langle k-l \rangle^{m + \gamma}\langle\frac{1}{k-l}\rangle^\epsilon \lesssim \langle k-l-\eta\rangle^{m + \gamma}\langle\frac{1}{k-l-\eta}\rangle^\epsilon,
    	\end{align*}
    	one gets
    	\begin{align*}
    		I_1 &\lesssim \big\|  e^{a A^{-\frac{1}{3}}|k-l|^{\frac{2}{3} } t } \langle k-l-\eta\rangle^{m + \gamma}\langle\frac{1}{k-l-\eta}\rangle^\epsilon   \| n_{k-l-\eta}\|_{L_y^2} \|n_{\eta}\|_{L_y^\infty}   \big\|_{L_{k-l}^2 L_\eta^1(\mathbb{R}^2)}\\
    		&\lesssim \big\|  e^{a A^{-\frac{1}{3}}|k-l-\eta|^{\frac{2}{3} } t } \langle k-l-\eta\rangle^{m + \gamma}\langle\frac{1}{k-l-\eta}\rangle^\epsilon   \| n_{k-l-\eta}\|_{L_y^2}     \big\|_{L_{k-l-\eta}^2 }       \big\|  e^{a A^{-\frac{1}{3}}|\eta|^{\frac{2}{3} } t }    \|n_{\eta}\|_{L_y^\infty}   \big\|_{L_{\eta}^1 }\\
    		&\lesssim \big\|  e^{a A^{-\frac{1}{3}}|D_x|^{\frac{2}{3} } t } \langle D_x\rangle^{m + \gamma}\langle\frac{1}{D_x}\rangle^\epsilon   n    \|_{L^2 }   \||\eta|^{-\frac16}\langle \eta\rangle^{-m}\langle\frac{1}{\eta}\rangle^{-\epsilon}\big\|_{L_{\eta}^2 }\\
    		&\quad\times \big\|  e^{a A^{-\frac{1}{3}}|\eta|^{\frac{2}{3} } t }    \langle \eta\rangle^m\langle\frac{1}{\eta}\rangle^\epsilon \||\eta|^{\frac13} n_{\eta}\|_{L_y^2}^\frac12  \|\py n_{\eta}\|_{L_y^2}^\frac12   \big\|_{L_{\eta}^2 }\\
    		&\lesssim  \big\|  e^{a A^{-\frac{1}{3}}|D_x|^{\frac{2}{3} } t } \langle D_x\rangle^{m + \gamma}\langle\frac{1}{D_x}\rangle^\epsilon   n    \big\|_{L^2 }
    		\big\|  e^{a A^{-\frac{1}{3}}|D_x|^{\frac{2}{3} } t } \langle D_x\rangle^m\langle\frac{1}{D_x}\rangle^\epsilon |D_x|^\frac13  n    \big\|_{L^2 }^\frac12 \\
    		&\quad\times \big\|  e^{a A^{-\frac{1}{3}}|D_x|^{\frac{2}{3} } t } \langle D_x\rangle^m\langle\frac{1}{D_x}\rangle^\epsilon  \py n    \big\|_{L^2 }^\frac12 \alabel{eq:I_1}.
    	\end{align*}
    	For $I_2$, note that $|k-l| \sim |l|$. Then, based on \eqref{eq:GN1},
    	\begin{align*}
    		\langle k-l \rangle^{m + \gamma}\langle\frac{1}{k-l}\rangle^\epsilon \lesssim \langle \eta\rangle^{m + \gamma}\langle\frac{1}{\eta}\rangle^\epsilon,
    	\end{align*}
    	and $1\lesssim|k-l-\eta|^{-\frac16}|\eta|^\frac16$, we have
    	\begin{align*}
    		I_2 &\lesssim \big\|  e^{a A^{-\frac{1}{3}}|k-l|^{\frac{2}{3} } t } \langle  \eta\rangle^{m + \gamma}\langle\frac{1}{ \eta}\rangle^\epsilon   \| n_{k-l-\eta}\|_{L_y^2} \|n_{\eta}\|_{L_y^\infty}   \big\|_{L_{k-l}^2 L_\eta^1(\mathbb{R}^2)}\\
    		&\lesssim \big\|  e^{a A^{-\frac{1}{3}}|k-l-\eta|^{\frac{2}{3} } t }  |k-l-\eta|^{-\frac16}  \| n_{k-l-\eta}\|_{L_y^2}     \big\|_{L_{k-l-\eta}^1 }\big\|  e^{a A^{-\frac{1}{3}}|\eta|^{\frac{2}{3} } t } \langle  \eta\rangle^{m + \gamma}\langle\frac{1}{ \eta}\rangle^\epsilon  |\eta|^\frac16 \|n_{\eta}\|_{L_y^\infty}   \big\|_{L_{\eta}^2 }\\
    		&\lesssim \big\|  e^{a A^{-\frac{1}{3}}|D_x|^{\frac{2}{3} } t } \langle D_x\rangle^{m + \gamma}\langle\frac{1}{D_x}\rangle^\epsilon |D_x|^\frac13  n    \big\|_{L^2 }^\frac12 \big\|  e^{a A^{-\frac{1}{3}}|D_x|^{\frac{2}{3} } t } \langle D_x\rangle^{m + \gamma}\langle\frac{1}{D_x}\rangle^\epsilon  \py n    \big\|_{L^2 }^\frac12 \\
    		&\quad\times \big\|  e^{a A^{-\frac{1}{3}}|k-l-\eta|^{\frac{2}{3} } t } \langle k-l-\eta\rangle^m\langle\frac{1}{k-l-\eta}\rangle^\epsilon   \| n_{k-l-\eta}\|_{L_y^2}     \big\|_{L_{k-l-\eta}^2 } \\
    		&\quad\times \big\||k-l-\eta|^{-\frac16} \langle k-l-\eta\rangle^{-m}\langle\frac{1}{k-l-\eta}\rangle^{-\epsilon}  \big\|_{L_{k-l-\eta}^2 } \\
    		&\lesssim  \big\|  e^{a A^{-\frac{1}{3}}|D_x|^{\frac{2}{3} } t } \langle D_x\rangle^m\langle\frac{1}{D_x}\rangle^\epsilon   n    \big\|_{L^2 }
    		\big\|  e^{a A^{-\frac{1}{3}}|D_x|^{\frac{2}{3} } t } \langle D_x\rangle^{m + \gamma}\langle\frac{1}{D_x}\rangle^\epsilon |D_x|^\frac13  n    \big\|_{L^2 }^\frac12 \\
    		&\quad\times \big\|  e^{a A^{-\frac{1}{3}}|D_x|^{\frac{2}{3} } t } \langle D_x\rangle^{m + \gamma}\langle\frac{1}{D_x}\rangle^\epsilon  \py n    \big\|_{L^2 }^\frac12. \alabel{eq:I_2}
    	\end{align*}
    	At last, we estimate $I_3$. It follows from \eqref{eq:GN1}, 
        $$
        \langle k-l\rangle^{m + \gamma} \lesssim \langle k-l -\eta \rangle^{m + \gamma}\langle \eta\rangle^m\langle k-l\rangle^{-m},
        $$
        and
    	\begin{align*} 
    		1\leq  \langle \frac{1}{k-l-\eta} \rangle^\frac16 |k-l-\eta|^{\frac16} \lesssim \langle\frac{1}{\eta}\rangle^\epsilon\langle\frac{1}{k-l-\eta}\rangle^\epsilon\left(1 +\langle\frac{1}{k-l}\rangle^{\frac{1}{6}-2 \epsilon} \right)|\eta|^{\frac16}
    	\end{align*} 
    	that
    	\begin{align*}
    		I_3 &\lesssim \big\|  e^{a A^{-\frac{1}{3}}|k-l|^{\frac{2}{3} } t }\langle k-l -\eta \rangle^{m + \gamma}\langle \eta\rangle^m\langle k-l\rangle^{-m}\\
    		&\quad\times  \langle\frac{1}{\eta}\rangle^\epsilon\langle\frac{1}{k-l-\eta}\rangle^\epsilon\left(\langle\frac{1}{k-l}\rangle^{\epsilon} +\langle\frac{1}{k-l}\rangle^{\frac{1}{6}- \epsilon} \right)
    		|\eta|^{\frac16} \| n_{k-l-\eta}\|_{L_y^2} \|n_{\eta}\|_{L_y^\infty}   \big\|_{L_{k-l}^2 L_\eta^1(\mathbb{R}^2)}\\
    		&\lesssim \big\|  e^{a A^{-\frac{1}{3}}|k-l-\eta|^{\frac{2}{3} } t }  \langle k-l -\eta \rangle^{m + \gamma} \langle\frac{1}{k-l-\eta}\rangle^\epsilon \| n_{k-l-\eta}\|_{L_y^2}     \big\|_{L_{k-l-\eta}^2 }\big\|  e^{a A^{-\frac{1}{3}}|\eta|^{\frac{2}{3} } t } \langle  \eta\rangle^m\langle\frac{1}{ \eta}\rangle^\epsilon  |\eta|^\frac16 \|n_{\eta}\|_{L_y^\infty}   \big\|_{L_{\eta}^2 }\\
    		&\quad\times  \big\| \left(\langle\frac{1}{k-l}\rangle^{\epsilon} +\langle\frac{1}{k-l}\rangle^{\frac{1}{6}- \epsilon} \right) \langle k-l\rangle^{-m}\big\|_{L_{k-l}^2}\\
    		&\lesssim  \big\|  e^{a A^{-\frac{1}{3}}|D_x|^{\frac{2}{3} } t } \langle D_x\rangle^{m + \gamma}\langle\frac{1}{D_x}\rangle^\epsilon   n    \big\|_{L^2 }
    		\big\|  e^{a A^{-\frac{1}{3}}|D_x|^{\frac{2}{3} } t } \langle D_x\rangle^m\langle\frac{1}{D_x}\rangle^\epsilon |D_x|^\frac13  n    \big\|_{L^2 }^\frac12 \\
    		&\quad\times \big\|  e^{a A^{-\frac{1}{3}}|D_x|^{\frac{2}{3} } t } \langle D_x\rangle^m\langle\frac{1}{D_x}\rangle^\epsilon  \py n    \big\|_{L^2 }^\frac12 \alabel{eq:I_3}.
    	\end{align*}
    	 Combining \eqref{eq:I_1}--\eqref{eq:I_3} with \eqref{eq:n^2}, one completes the proof.
    \end{proof}
    
    \subsection{The energy estimate of $n$.} 
    For the PKSNS system near a Couette flow \eqref{eq:main}, we have the following estimate.
    
    \begin{Lem} \label{lem:est of n}
    	Assume that $0\leq t \leq T$ and $0<a<\frac{1}{16(1+2 \pi)}$. The estimate
    	\begin{align*} 
    			&\quad\|  \langle D_x\rangle^{m+\frac13}  \langle \frac{1}{D_x}\rangle^\epsilon   n  \|_{X_a}^2 \\
                &\lesssim  \|  \langle D_x\rangle^{m +\frac13}  \langle \frac{1}{D_x}\rangle^\epsilon   n_{\mathrm{in}} \|_{L^2}^2 +  \|  \langle D_x\rangle^{m+\frac13}  \langle \frac{1}{D_x}\rangle^\epsilon   n \|_{X_a}^2  \\
    			& \qquad \times \lt(   \frac{1}{A^\frac12} \|  \langle D_x\rangle^m  \langle \frac{1}{D_x}\rangle^\epsilon   \omega \|_{X_a}  + \frac{|\mu|}{A^\frac13}  \|  \langle D_x\rangle^{m+\frac13}   \langle \frac{1}{D_x}\rangle^\epsilon   n \|_{X_a} + \frac{|\kappa|}{A^\frac16}  \|  \langle D_x\rangle^{m+\frac13}   \langle \frac{1}{D_x}\rangle^\epsilon   n \|_{X_a}^2 \rt)
    	\end{align*}
        holds under either of the following assumptions:
    \begin{enumerate}
        \item[(i)] $\alpha > 0$ and $0<\epsilon<\frac12<m$;
        \item[(ii)] $\alpha = 0$ and $\frac13<\epsilon<\frac12<m$.
    \end{enumerate}
    \end{Lem}
    
    \begin{proof}
    	Note that
    	\begin{equation}\begin{aligned} \label{eq:temp0.12}
    			\|  \langle D_x\rangle^{m+\frac13}  \langle \frac{1}{D_x}\rangle^\epsilon   n  \|_{X_a}^2 \lesssim \|  \langle D_x\rangle^{m }  \langle \frac{1}{D_x}\rangle^\epsilon   n  \|_{X_a}^2 + \|  \langle D_x\rangle^{m }  \langle \frac{1}{D_x}\rangle^\epsilon  |D_x|^\frac13 n  \|_{X_a}^2.
    	\end{aligned}\end{equation}
    	One applies $|D_x|^\frac13$ to $\eqref{eq:main}_1$ and has that
    	\begin{equation}\begin{aligned} \label{eq:dx 13 n}
    			\partial_t |D_x|^\frac13  n+y \partial_x |D_x|^\frac13  n-\frac{1}{A} \Delta |D_x|^\frac13  n=-\frac{1}{A} |D_x|^\frac13   \nabla \cdot  \lt[n \lt(\kappa n +\mu\rt) \nabla c\rt]  - \frac1A |D_x|^\frac13  (u \cdot \nabla n).
    	\end{aligned}\end{equation}
    	Since 
    	\begin{align*}
    		\big(|D_x|^\frac13 (fg) \big| |D_x|^\frac13 h\big) &\lesssim   \big( |k| ^\frac13 \widehat{fg} \big|  |k| ^\frac13 \hat{h} \big) \\
    		&\lesssim \int_{\mathbb{R}^3}  |k-l|^\frac13|\hat{f}(k-l, \xi) \hat{g}(l, \xi)| |k| ^\frac13 |\hat{h}(k)| dkdld \xi\\
    		&\quad  +  \int_{\mathbb{R}^3}   |l|^ \frac13|\hat{f}(k-l, \xi) \hat{g}(l, \xi)| |k| ^\frac13 |\hat{h}(k)| dkdld\xi
    	\end{align*}
    	holds for some functions $f$, $g$ and $h$, we formally decompose $ |D_x|^\frac13  (u \cdot \nabla n) $ into $u \cdot \nabla |D_x|^\frac13  n $ and $| D_x |^\frac13  u \cdot \nabla n$ in $L^2$ energy estimate. Hence, by applying Proposition \ref{lem:est of f} to \eqref{eq:dx 13 n}, one has
    	\begin{align*} 
    			&\quad\|  \langle D_x\rangle^m  \langle \frac{1}{D_x}\rangle^\epsilon   |D_x|^\frac13  n  \|_{X_a}^2 \\ &\lesssim   \|  \langle D_x\rangle^m  \langle \frac{1}{D_x}\rangle^\epsilon   |D_x|^\frac13  n_{\mathrm{in}} \|_{L^2}^2 + 	\frac{1}{A^\frac12}  \|  \langle D_x\rangle^m  \langle \frac{1}{D_x}\rangle^\epsilon   \omega \|_{X_a}   \|  \langle D_x\rangle^m  \langle \frac{1}{D_x}\rangle^\epsilon   |D_x|^\frac13  n \|_{X_a}^2 \\
    			&\quad +  \frac1A \int_{0}^{t} \left|\Re\left(  | D_x |^\frac13  u \cdot \nabla n \left\lvert\, \mathcal{M} e^{2 a A^{-\frac{1}{3}}\left|D_x\right|^{\frac{2}{3} }t }\langle D_x\rangle^{2 m}\langle\frac{1}{D_x}\rangle^{2 \epsilon} |D_x|^\frac13  n \right.\right)\right| dt \\
                &\quad + \frac1A \int_{0}^{t} \left|\Re\left( |D_x|^\frac13  (\nabla \cdot(\mu n \nabla c) ) \left\lvert\, \mathcal{M} e^{2 a A^{-\frac{1}{3}}\left|D_x\right|^{\frac{2}{3} }t }\langle D_x\rangle^{2 m}\langle\frac{1}{D_x}\rangle^{2 \epsilon} |D_x|^\frac13  n\right.\right)\right| dt\\
        			&\quad + \frac1A \int_{0}^{t} \left|\Re\left( |D_x|^\frac13  (\nabla \cdot(\kappa n^2 \nabla c) ) \left\lvert\, \mathcal{M} e^{2 a A^{-\frac{1}{3}}\left|D_x\right|^{\frac{2}{3} }t }\langle D_x\rangle^{2 m}\langle\frac{1}{D_x}\rangle^{2 \epsilon} |D_x|^\frac13  n\right.\right)\right| dt.\alabel{eq:temp001}
    	\end{align*}
    	Using Lemma \ref{lem:est of dx u}, one gets
    	\begin{equation}\begin{aligned}   \label{eq:temp002} 
    			&\quad\frac1A \int_{0}^{t} \left|\Re\left(|D_x|^{\frac13}u \cdot \nabla n \left\lvert\, \mathcal{M} e^{2 a A^{-\frac{1}{3}}\left|D_x\right|^{\frac{2}{3} }t }\langle D_x\rangle^{2 m   }\langle\frac{1}{D_x}\rangle^{2 \epsilon} |D_x|^\frac13  n  \right.\right)\right|   dt  \\
    			&\lesssim \frac1{A^\frac12} \| \langle D_x\rangle^m\langle\frac{1}{D_x}\rangle^\epsilon \omega \|_{X_a}    \| \langle D_x\rangle^m\langle\frac{1}{D_x}\rangle^\epsilon |D_x|^\frac13 n \|_{X_a}^2.
    	\end{aligned}\end{equation}
    	Similarly, applying Proposition \ref{lem:est of f} to $\eqref{eq:main}_1$, we have
    	\begin{align*} 
    			&\quad\|  \langle D_x\rangle^m  \langle \frac{1}{D_x}\rangle^\epsilon     n  \|_{X_a}^2 \\ &\lesssim   \|  \langle D_x\rangle^m  \langle \frac{1}{D_x}\rangle^\epsilon     n_{\mathrm{in}} \|_{L^2}^2 + 	\frac{1}{A^\frac12}  \|  \langle D_x\rangle^m  \langle \frac{1}{D_x}\rangle^\epsilon   \omega \|_{X_a}   \|  \langle D_x\rangle^m  \langle \frac{1}{D_x}\rangle^\epsilon      n \|_{X_a}^2 \\
                &\quad + \frac1A \int_{0}^{t} \left|\Re\left(  \nabla \cdot(\mu n \nabla c)   \left\lvert\, \mathcal{M} e^{2 a A^{-\frac{1}{3}}\left|D_x\right|^{\frac{2}{3} }t }\langle D_x\rangle^{2 m}\langle\frac{1}{D_x}\rangle^{2 \epsilon}   n\right.\right)\right| dt\\
    			&\quad + \frac1A \int_{0}^{t} \left|\Re\left(  \nabla \cdot(\kappa n^2 \nabla c)   \left\lvert\, \mathcal{M} e^{2 a A^{-\frac{1}{3}}\left|D_x\right|^{\frac{2}{3} }t }\langle D_x\rangle^{2 m}\langle\frac{1}{D_x}\rangle^{2 \epsilon}   n\right.\right)\right| dt.
                \alabel{eq:temp003}
    	\end{align*}
        Regarding the term $\nabla\cdot (n\nabla c)$, we invoke (3.6) (for $\alpha>0$) and (3.8) (for $\alpha = 0$) from \cite{CWY2025b}, and  correspondingly under the conditions $0<\epsilon<\frac12<m$ and $\frac13<\epsilon<\frac12<m$, we obtain
\begin{align}
    \frac1A \int_{0}^{t} \left|\Re\left(  \nabla \cdot(\mu n \nabla c)   \left\lvert\, \mathcal{M} e^{2 a A^{-\frac{1}{3}}\left|D_x\right|^{\frac{2}{3} }t }\langle D_x\rangle^{2 m}\langle\frac{1}{D_x}\rangle^{2 \epsilon}   n\right.\right)\right| dt \lesssim \frac{|\mu|}{A^\frac13} \|  \langle D_x\rangle^m  \langle \frac{1}{D_x}\rangle^\epsilon     n  \|_{X_a}^3.
\end{align}
Here, the $X_a$-norm follows the definition in \cite[(1.6)]{CWY2025b} with parameters $\Xi = 2\pi$, $\theta_1 = \frac14$, $\theta_2 =\frac12$ and $0<a<\frac{1}{16(1+2 \pi)}$. Similarly, for $|D_x|^\frac13 \nabla\cdot (n\nabla c)$, utilizing (4.9) and (4.10) in the same reference yields
\begin{align}  
    \frac1A \int_{0}^{t} \left|\Re\left(  |D_x|^\frac13 \nabla \cdot(\mu n \nabla c)   \left\lvert\, \mathcal{M} e^{2 a A^{-\frac{1}{3}}\left|D_x\right|^{\frac{2}{3} }t }\langle D_x\rangle^{2 m}\langle\frac{1}{D_x}\rangle^{2 \epsilon} |D_x|^\frac13  n\right.\right)\right| dt \lesssim \frac{|\mu|}{A^\frac13} \|  \langle D_x\rangle^{m+ \frac13} \langle \frac{1}{D_x}\rangle^\epsilon     n  \|_{X_a}^3.
    \label{eq:temp004}
\end{align}

    	For simplicity, we combine the last term of \eqref{eq:temp001} and the last term of \eqref{eq:temp003} into a single term for estimation as follows:
    	\begin{align*} 
    			&\quad \frac1A\left|\Re\left( \nabla\cdot(\kappa n^2\nabla c) \left\lvert\, \mathcal{M} e^{2 a A^{-\frac{1}{3}}\left|D_x\right|^{\frac{2}{3} }t }\langle D_x\rangle^{2 m   }\langle\frac{1}{D_x}\rangle^{2 \epsilon} n\right.\right)\right| \\
    			& \quad +  \frac1A \left|\Re\left( |D_x|^{\frac13} ( \nabla\cdot(\kappa n^2\nabla c) ) \left\lvert\, \mathcal{M} e^{2 a A^{-\frac{1}{3}}\left|D_x\right|^{\frac{2}{3} }t }\langle D_x\rangle^{2 m  }\langle\frac{1}{D_x}\rangle^{2 \epsilon} |D_x|^{\frac13} n\right.\right)\right|\\
    			& \lesssim\frac1A \left|\Re\left(   \nabla\cdot(\kappa n^2 \nabla c) \left\lvert\, \mathcal{M} e^{2 a A^{-\frac{1}{3}}\left|D_x\right|^{\frac{2}{3} }t }\langle D_x\rangle^{2 m + \frac23 }\langle\frac{1}{D_x}\rangle^{2 \epsilon}    n\right.\right)\right|\\
    			&\lesssim \frac{|\kappa|} A \int_{\mathbb{R}^3} e^{2 a A^{-\frac{1}{3}}|k|^{\frac{2}{3} } t }\langle k\rangle^{2 m +\frac23 }\langle\frac{1}{k}\rangle^{2 \epsilon}\left|\int_{\mathbb{R}} \mathcal{M}\left(k, D_y\right) \nabla_k n_k(y) \cdot\nabla_l c_l(y)  n_{k-l-\eta}(y) n_{\eta}(y)   d y\right| d k d ld\eta . \alabel{eq:Fourier}
    	\end{align*}
    	
    	In the sequel, using Lemma \ref{lem:Elliptic estimate}, we formally estimate $\| \nabla^2 c\|_{L^2}$ via $\| n\|_{L^2}$ for both the $\alpha > 0$ and $\alpha = 0$ cases. Furthermore, we partition the integration domain into three regions, as outlined in Remark \ref{rem:domain}.
    	
    	$\bullet$~Case 1:   $\frac{|k-l|}{2} \leq |k| \leq 2|k-l|$. By applying
    	\begin{align*}
    		\langle k\rangle^{m+\frac13}\langle\frac{1}{k}\rangle^\epsilon \lesssim \langle k-l\rangle^{m+\frac13}\langle\frac{1}{k-l}\rangle^\epsilon,
    	\end{align*}
        we have
    	\begin{align*}  
    			&\quad \int_{\frac{|k-l|}{2} \leq|k| \leq 2|k-l|} e^{2 a A^{-\frac{1}{3}}|k|^{\frac{2}{3} } t }\langle k\rangle^{2 m +\frac23 } \langle\frac{1}{k}\rangle^{2 \epsilon}   \left|   \int_{\mathbb{R}} \mathcal{M}\left(k, D_y\right) \nabla_k n_k(y) \cdot\nabla_l c_l(y)   n_{k-l-\eta}(y) n_{\eta}(y) d y\right| d k d l d\eta \\
    			&\lesssim \big\|e^{a A^{-\frac{1}{3}}|k|^{\frac{2}{3} }t}\langle k\rangle^{m+\frac13 }\langle\frac{1}{k}\rangle^\epsilon   \|\nabla_k n_k\|_{L_y^2}\big\|_{L_k^2}     \big\|e^{a A^{-\frac{1}{3}}|l|^{\frac{2}{3} } t } \| \nabla_l c_l \|_{L_y^{\infty}}\big\|_{L_l^1} \\
    			&\quad \times\big\|e^{a A^{-\frac{1}{3}}|k-l|^{\frac{2}{3} } t }\langle k-l\rangle^{m +\frac13 } \langle\frac{1}{k-l}\rangle^\epsilon    \int_{\mathbb{R}}\| n_{k-l-\eta}\|_{L_y^2} \|n_{\eta}\|_{L_y^\infty} d\eta \big\|_{L_{k-l}^2} \\
    			&\lesssim \big\|e^{a A^{-\frac{1}{3}}\left|D_x\right|^{\frac{2}{3} } t }\langle D_x\rangle^{m +\frac13 }\langle\frac{1}{D_x}\rangle^\epsilon \nabla n \big\|_{L^2}   \big\|e^{a A^{-\frac{1}{3}}\left|D_x\right|^{\frac{2}{3} } t }\langle D_x\rangle^m\langle\frac{1}{D_x}\rangle^\epsilon n\big\|_{L^2}     \\
    			&\quad \times  \big\|e^{a A^{-\frac{1}{3}}|k-l|^{\frac{2}{3} } t }\langle k-l\rangle^{m +\frac13  } \langle\frac{1}{k-l}\rangle^\epsilon    \int_{\mathbb{R}}\| n_{k-l-\eta}\|_{L_y^2} \|n_{\eta}\|_{L_y^\infty} d\eta \big\|_{L_{k-l}^2} , \alabel{eq:case1}
    	\end{align*}
    	where we used \eqref{eq:GN2} and Lemma \ref{lem:Elliptic estimate} to get
    	\begin{align*}
    			 \big\|e^{a A^{-\frac{1}{3}}|l|^{\frac{2}{3} } t }     \| \nabla_l c_l \|_{L_y^{\infty}}\big\|_{L_l^1} 
    			& \leq \big\|e^{a A^{-\frac{1}{3}}|l|^{\frac{2}{3} }t}\langle l\rangle^m\langle\frac{1}{l}\rangle^\epsilon   |l|^\frac12 \|   \nabla_l c_l \|_{L_y^\infty}  \big\|_{L_l^2}\big\| |l|^{-\frac12} \langle l\rangle^{-m}\langle\frac{1}{l}\rangle^{-\epsilon} \big\|_{L_l^2}  \\
    			&\lesssim  \|e^{a A^{-\frac{1}{3}}\left|D_x\right|^{\frac{2}{3} }t}\langle D_x\rangle^m\langle\frac{1}{D_x}\rangle^\epsilon  n\|_{L^2} \alabel{eq:temp0.341}.
    	\end{align*}
    	Applying Lemma \ref{lem:a preliminary step} (choosing $\gamma=\frac13$) to \eqref{eq:case1}, we get
    	\begin{align*}  
    		&\quad \int_{\frac{|k-l|}{2} \leq|k| \leq 2|k-l|} e^{2 a A^{-\frac{1}{3}}|k|^{\frac{2}{3} } t }\langle k\rangle^{2 m +\frac23 } \langle\frac{1}{k}\rangle^{2 \epsilon}   \left|   \int_{\mathbb{R}} \mathcal{M}\left(k, D_y\right) \nabla_k n_k(y) \cdot\nabla_l c_l(y)   n_{k-l-\eta}(y) n_{\eta}(y) d y\right| d k d l d\eta \\
    		&\lesssim \big\|e^{a A^{-\frac{1}{3}}\left|D_x\right|^{\frac{2}{3} } t }\langle D_x\rangle^{m +\frac13 }\langle\frac{1}{D_x}\rangle^\epsilon \nabla n \big\|_{L^2}   \big\|e^{a A^{-\frac{1}{3}}\left|D_x\right|^{\frac{2}{3} } t }\langle D_x\rangle^m\langle\frac{1}{D_x}\rangle^\epsilon n\big\|_{L^2}  \big\|e^{a A^{-\frac{1}{3}}\left|D_x\right|^{\frac{2}{3} } t }\langle D_x\rangle^{m+\frac13}\langle\frac{1}{D_x}\rangle^\epsilon n\big\|_{L^2}    \\
    		&\quad \times  \big\|  e^{a A^{-\frac{1}{3}}|D_x|^{\frac{2}{3} } t } \langle D_x\rangle^{m+\frac13}\langle\frac{1}{D_x}\rangle^\epsilon |D_x|^\frac13  n    \big\|_{L^2 }^\frac12 \big\|  e^{a A^{-\frac{1}{3}}|D_x|^{\frac{2}{3} } t } \langle D_x\rangle^{m+\frac13}\langle\frac{1}{D_x}\rangle^\epsilon  \py n    \big\|_{L^2 }^\frac12 . \alabel{case1}
    	\end{align*}

    	$\bullet$~Case 2:   $2|k-l|<|k|$.
    	By $1\lesssim|k-l|^{-\frac12} |l|^\frac12$, Lemma \ref{lem:a preliminary step} ($\gamma=0$) and
    	\begin{align*}
    		\langle k\rangle^{m+\frac13}\langle\frac{1}{k}\rangle^\epsilon \lesssim \langle l\rangle^{m+\frac13}\langle\frac{1}{l}\rangle^\epsilon,
    	\end{align*}
    	we get
    	\begin{align*}   
    			&\quad\int_{2|k-l|<|k|}e^{2 a A^{-\frac{1}{3}}|k|^{\frac{2}{3} } t }\langle k\rangle^{2 m  +\frac23} \langle\frac{1}{k}\rangle^{2 \epsilon}   \left|   \int_{\mathbb{R}} \mathcal{M}\left(k, D_y\right) \nabla_k n_k(y) \cdot\nabla_l c_l(y)   n_{k-l-\eta}(y) n_{\eta}(y) d y\right| d k d l d\eta  \\
    			&\lesssim  \big\|e^{a A^{-\frac{1}{3}}|k|^{\frac{2}{3} }t}\langle k\rangle^{m +\frac13 }\langle\frac{1}{k}\rangle^\epsilon  \| \nabla_k n_k\|_{L_y^2}\big\|_{L_k^2}     \big\|e^{a A^{-\frac{1}{3}}|l|^{\frac{2}{3} } t } \langle l\rangle^{m  +\frac13  }\langle\frac{1}{l}\rangle^\epsilon |l|^\frac12 \| \nabla_l c_l \|_{L_y^{\infty}}\big\|_{L_l^2} \\
    			&\quad \times\big\|e^{a A^{-\frac{1}{3}}|k-l|^{\frac{2}{3} } t }\langle k-l\rangle^{m  } \langle\frac{1}{k-l}\rangle^\epsilon    \int_{\mathbb{R}}\| n_{k-l-\eta}\|_{L_y^2} \|n_{\eta}\|_{L_y^\infty} d\eta \big\|_{L_{k-l}^2} 
    			\big\||k-l|^{-\frac12}\langle k-l\rangle^{-m}\langle\frac{1}{k-l}\rangle^{-\epsilon}\big\|_{L_{k-l}^2}\\
    			&\lesssim \big\|e^{a A^{-\frac{1}{3}}\left|D_x\right|^{\frac{2}{3} } t }\langle D_x\rangle^{m +\frac13  }\langle\frac{1}{D_x}\rangle^\epsilon \nabla n \big\|_{L^2}   \big\|e^{a A^{-\frac{1}{3}}\left|D_x\right|^{\frac{2}{3} } t }\langle D_x\rangle^{m+\frac13}\langle\frac{1}{D_x}\rangle^\epsilon n\big\|_{L^2}  \big\|e^{a A^{-\frac{1}{3}}\left|D_x\right|^{\frac{2}{3} } t }\langle D_x\rangle^m\langle\frac{1}{D_x}\rangle^\epsilon n\big\|_{L^2}    \\
    			&\quad \times  \big\|  e^{a A^{-\frac{1}{3}}|D_x|^{\frac{2}{3} } t } \langle D_x\rangle^m\langle\frac{1}{D_x}\rangle^\epsilon |D_x|^\frac13  n    \big\|_{L^2 }^\frac12 \big\|  e^{a A^{-\frac{1}{3}}|D_x|^{\frac{2}{3} } t } \langle D_x\rangle^m\langle\frac{1}{D_x}\rangle^\epsilon  \py n    \big\|_{L^2 }^\frac12, \alabel{case2}
    	\end{align*}
    	where we used $\epsilon>0$ and $m>0$.
    	
    	$\bullet$~Case 3:   $2|k|<|k-l|$.
    	Using Lemma \ref{lem:a preliminary step} ($\gamma=\frac13$),
        $$
        \langle k \rangle^{2m +\frac23} \leq  \langle l \rangle^{m+\frac13} \langle k-l \rangle^{m+\frac13},
        $$
        and
        \begin{align*} 
    			1 \leq \langle \frac{1}{k-l} \rangle^\frac12 |k-l|^{\frac12} \lesssim \langle\frac{1}{l}\rangle^\epsilon\langle\frac{1}{k-l}\rangle^\epsilon\left(1+\langle\frac{1}{k}\rangle^{\frac{1}{2}-2 \epsilon} \right)|l|^{\frac12},
    	\end{align*}
    	when $0<\epsilon<\frac12<m$, we obtain
    	\begin{align*}  
    			&\quad\int_{2|k|<|k-l|} e^{2 a A^{-\frac{1}{3}}|k|^{\frac{2}{3} } t }\langle k\rangle^{2 m +\frac23 } \langle\frac{1}{k}\rangle^{2 \epsilon}   \left|   \int_{\mathbb{R}} \mathcal{M}\left(k, D_y\right) \nabla_k n_k(y) \cdot\nabla_l c_l(y)   n_{k-l-\eta}(y) n_{\eta}(y) d y\right| d k d l d\eta\\
    			& \lesssim \big\|e^{a A^{-\frac{1}{3}}|k|^{\frac{2}{3} } t }\langle k\rangle^{m }\langle\frac{1}{k}\rangle^\epsilon   \|\nabla_k  n_k\|_{L_y^2}\big\|_{L_k^2}  \big\|e^{a A^{-\frac{1}{3}}|l|^{\frac{2}{3} } t }\langle l\rangle^{m +\frac13  }\langle\frac{1}{l}\rangle^\epsilon  |l|^{\frac12}  \|\nabla_l  c_l\|_{L_y^\infty}\big\|_{L_l^2} \\
    			& \quad \times   \big\|e^{a A^{-\frac{1}{3}}|k-l|^{\frac{2}{3} } t }\langle k-l\rangle^{m  +\frac13 } \langle\frac{1}{k-l}\rangle^\epsilon    \int_{\mathbb{R}}\| n_{k-l-\eta}\|_{L_y^2} \|n_{\eta}\|_{L_y^\infty} d\eta \big\|_{L_{k-l}^2} 
    			\big\|  \left(\langle\frac{1}{k}\rangle^\epsilon  +\langle\frac{1}{k}\rangle^{\frac12 -\epsilon}  \right)\langle k\rangle^{-m}\big\|_{L_k^2}\\
    			&\lesssim \big\|e^{a A^{-\frac{1}{3}}\left|D_x\right|^{\frac{2}{3} } t }\langle D_x\rangle^{m  }\langle\frac{1}{D_x}\rangle^\epsilon \nabla n \big\|_{L^2}   \big\|e^{a A^{-\frac{1}{3}}\left|D_x\right|^{\frac{2}{3} } t }\langle D_x\rangle^{m+\frac13}\langle\frac{1}{D_x}\rangle^\epsilon n\big\|_{L^2}^2     \\
    			&\quad \times  \big\|  e^{a A^{-\frac{1}{3}}|D_x|^{\frac{2}{3} } t } \langle D_x\rangle^{m+\frac13}\langle\frac{1}{D_x}\rangle^\epsilon |D_x|^\frac13  n    \big\|_{L^2 }^\frac12 \big\|  e^{a A^{-\frac{1}{3}}|D_x|^{\frac{2}{3} } t } \langle D_x\rangle^{m+\frac13}\langle\frac{1}{D_x}\rangle^\epsilon  \py n    \big\|_{L^2 }^\frac12. \alabel{case3}
    	\end{align*}
    	
    	Collecting the estimates \eqref{case1}--\eqref{case3} and using \eqref{eq:Fourier}, we arrive at
    	\begin{equation*}\begin{aligned}
    			&\quad\left|\Re\left(\nabla \cdot(\kappa n^2\nabla c) \left\lvert\, \mathcal{M} e^{2 a A^{-\frac{1}{3}}\left|D_x\right|^{\frac{2}{3} }t }\langle D_x\rangle^{2 m +\frac23 }\langle\frac{1}{D_x}\rangle^{2 \epsilon} n\right.\right)\right| \\
    			&\lesssim |\kappa| \big\|e^{a A^{-\frac{1}{3}}\left|D_x\right|^{\frac{2}{3} } t }\langle D_x\rangle^{m+\frac13} \langle\frac{1}{D_x}\rangle^\epsilon \nabla n \big\|_{L^2}   \big\|e^{a A^{-\frac{1}{3}}\left|D_x\right|^{\frac{2}{3} } t }\langle D_x\rangle^{m+\frac13}\langle\frac{1}{D_x}\rangle^\epsilon n\big\|_{L^2}^2     \\
    			&\quad \times  \big\|  e^{a A^{-\frac{1}{3}}|D_x|^{\frac{2}{3} } t } \langle D_x\rangle^{m+\frac13}\langle\frac{1}{D_x}\rangle^\epsilon |D_x|^\frac13  n    \big\|_{L^2 }^\frac12 \big\|  e^{a A^{-\frac{1}{3}}|D_x|^{\frac{2}{3} } t } \langle D_x\rangle^{m+\frac13}\langle\frac{1}{D_x}\rangle^\epsilon  \py n    \big\|_{L^2 }^\frac12, 
    	\end{aligned}\end{equation*}
    	which implies 
    	\begin{align*} 
    			&\quad \frac1A \int_{0}^{t}  \left|\Re\left( \nabla\cdot(\kappa n^2\nabla c) \left\lvert\, \mathcal{M} e^{2 a A^{-\frac{1}{3}}\left|D_x\right|^{\frac{2}{3} }t }\langle D_x\rangle^{2 m   }\langle\frac{1}{D_x}\rangle^{2 \epsilon} n\right.\right)\right| dt \\
    			& \quad +  \frac1A \int_{0}^{t} \left|\Re\left( |D_x|^{\frac13} ( \nabla\cdot(\kappa n^2\nabla c) ) \left\lvert\, \mathcal{M} e^{2 a A^{-\frac{1}{3}}\left|D_x\right|^{\frac{2}{3} }t }\langle D_x\rangle^{2 m  }\langle\frac{1}{D_x}\rangle^{2 \epsilon} |D_x|^{\frac13} n\right.\right)\right|dt \\ 
    			&\lesssim \frac{|\kappa| }{A^\frac16} \|  \langle D_x\rangle^{m+\frac13} \langle \frac{1}{D_x}\rangle^\epsilon   n \|_{X_a}^4. \alabel{eq:nabla n^2 nabla c}
    	\end{align*}
    	Combining \eqref{eq:temp001}--\eqref{eq:temp004} and \eqref{eq:nabla n^2 nabla c} with \eqref{eq:temp0.12}, we complete the proof.
    \end{proof}

    To improve the hypotheses \eqref{eq:bootstap}, it remains to complete the $L^\infty L^\infty$ estimate of $n$, which will be achieved by the Moser--Alikakos iteration.
    
    \begin{Lem} \label{lem:n L infty L infty}
    	Suppose that $0\leq t \leq T$, $0<a<\frac{1}{16(1+2 \pi)}$ and $0<\epsilon< 1/2<m$. Then it holds that\\
        {\bf Case of $\alpha>0$:} 
    	\begin{align*}
           &\quad\| n \|_{L^\infty L^\infty} \\ 
           &\leq C\lt(|\kappa|^3\|{n}\|_{L^{\infty}L^{2}}^3+\mu^4 \|\nabla c\|_{L^{\infty}L^{4}}^4  + \kappa^2 + 1\rt)\\
           &\qquad\times \lt( \|n_{\mathrm{in}}\|_{L^\infty}^2 +  \|n \|_{L^\infty L^2}^2 + \frac{1}{A^\frac13} \lt( |\kappa| \|n\|_{L^\infty L^\infty}^2 + |\mu| \|n\|_{L^\infty L^\infty} \rt)  \| \langle D_x \rangle^m \langle \frac{1}{D_x} \rangle^\epsilon n \|_{X_a}^2  +1\rt);
       \end{align*}
       {\bf Case of $\alpha=0$:}
    	\begin{align*}
           &\quad \| n \|_{L^\infty L^\infty} \\
           &\leq C\lt(|\kappa|^3\|{n}\|_{L^{\infty}L^{2}}^3+\mu^4 \|\nabla c\|_{L^{\infty}L^{4}}^4  + \kappa^2 + 1\rt)\\
           &\qquad\times \lt( \|n_{\mathrm{in}}\|_{L^\infty}^2 +  \|n \|_{L^\infty L^2}^2 + \frac{1}{A^\frac16} \lt( {|\kappa|}^\frac12 \|n\|_{L^\infty L^\infty} + {|\mu|}^\frac12 \|n\|_{L^\infty L^\infty}^\frac12 \rt)  \| \langle D_x \rangle^m \langle \frac{1}{D_x} \rangle^\epsilon n \|_{X_a}^2  +1\rt).
       \end{align*}
    \end{Lem}
    \begin{proof}
    	\underline{\bf Step I: $L^\infty L^4$ estimate of n.} 
        
        $\bullet$~Case of $\alpha>0$.
    	Multiplying $\eqref{eq:main}_1$ by $4 n^{3}$, integrating the resulting equation over $\mathbb{R}^2 $, and using $ \nabla\cdot u=0 $, one has
    	\begin{align*}
    		\frac{d}{d t}\|n^2\|_{L^2}^2+\frac{3}{ A}\|\nabla n^2\|_{L^2}^2 
    		&	=  \frac{6}{A} \int_{  \mathbb{R}^2 } n^2(\kappa n+\mu) \nabla c \cdot \nabla n^2 d x d y 
    		\leq  \frac{6}{A}\|n^2(\kappa n+\mu) \nabla c\|_{L^2}\|\nabla n^2\|_{L^2} \\
    		&\leq  \frac{3}{2 A}\|\nabla n^2\|_{L^2}^2+\frac{6}{A}\|n^2(\kappa n+\mu) \nabla c\|_{L^2}^2 \\
    		& \leq \frac{3}{2 A}\|\nabla n^2\|_{L^2}^2  + \frac{C}{A} \lt( \kappa^2 \|n\|_{L^\infty}^4 + \mu^2 \|n\|_{L^\infty}^2\rt)\| n\|_{L_x^\infty L_y^2}^2 \| \nabla c\|_{L_x^2 L_y^\infty}^2. \alabel{eq:L^infty L^4'}
    	\end{align*}
        Thanks to the Fourier inversion formula, we get
        \begin{align*}
            \| n (\cdot,y) \|_{L_x^\infty} & \leq   \| n_k(y) \|_{L_k^1} \leq \| \langle k \rangle^m \langle \frac{1}{k} \rangle^\epsilon |k|^\frac13 n_k(y) \|_{L_k^2} \| |k|^{-\frac13} \langle k \rangle^{-m} \langle \frac{1}{k} \rangle^{-\epsilon}   \|_{L_k^2}\\
            &\lesssim \| \langle D_x \rangle^m \langle \frac{1}{D_x} \rangle^\epsilon |D_x|^\frac13 n \|_{L_x^2},
        \end{align*}
        which implies 
        \begin{align*}
            \| n\|_{L_x^\infty L_y^2} \lesssim \| \langle D_x \rangle^m \langle \frac{1}{D_x} \rangle^\epsilon |D_x|^\frac13 n \|_{L^2}.
        \end{align*}
        Using Lemma \ref{lem:Elliptic estimate}, we infer that
        \begin{align*}
            \| \nabla c\|_{L_x^2 L_y^\infty} & =  \big\| \|\nabla c\|_{L_y^\infty} \big\|_{L_x^2 } \leq \big\| \|\nabla c\|_{L_y^2}^\frac12   \|\py \nabla c\|_{L_y^2}^\frac12 \big\|_{L_x^2 } \\
            &\leq \|\nabla c\|_{L^2}^\frac12 \|\py \nabla c\|_{L^2}^\frac12 \lesssim \|n\|_{L^2}.
        \end{align*}
        Hence, integrating \eqref{eq:L^infty L^4'} over $(0, t)$, we have
    	\begin{align*}
    		\|n\|_{L^\infty L^4}^4 &\leq \|n_{\mathrm{in}}\|_{L^4}^4 + \frac{C}{A}\lt( \kappa^2 \|n\|_{L^\infty L^\infty}^4 + \mu^2 \|n\|_{L^\infty L^\infty}^2 \rt) \| \langle D_x \rangle^m \langle \frac{1}{D_x} \rangle^\epsilon |D_x|^\frac13 n \|_{L^2L^2}^2 \| n \|_{L^\infty L^2}^2\\
            &\leq \frac12\|n_{\mathrm{in}}\|_{L^\infty}^4 + \frac12 \|n \|_{L^\infty L^2}^4 + \frac{C}{A^\frac23} \lt( \kappa^2 \|n\|_{L^\infty L^\infty}^4 + \mu^2 \|n\|_{L^\infty L^\infty}^2 \rt)   \| \langle D_x \rangle^m \langle \frac{1}{D_x} \rangle^\epsilon n \|_{X_a}^4. \alabel{eq:L^infty L^4}
    	\end{align*}

        $\bullet$~Case of $\alpha = 0$. One rewrites \eqref{eq:L^infty L^4'} as
        \begin{align*}
    		\frac{d}{d t}\|n^2\|_{L^2}^2+\frac{3}{ A}\|\nabla n^2\|_{L^2}^2 
    		&	=  \frac{6\kappa}{A} \int_{  \mathbb{R}^2 } n^3 \nabla c \cdot \nabla n^2 d x d y + \frac{6\mu}{A} \int_{  \mathbb{R}^2 } n^2 \nabla c \cdot \nabla n^2 d x d y\\ 
    		& =  -\frac{12\kappa}{5A} \int_{  \mathbb{R}^2 } n^5 \Delta c  d x d y -\frac{3\mu}{A} \int_{  \mathbb{R}^2 } n^4 \Delta c  d x d y   \\
            & =  \frac{12\kappa}{5A} \int_{  \mathbb{R}^2 } n^6  d x d y  + \frac{3\mu}{A} \int_{  \mathbb{R}^2 } n^5 d x d y   \\
    		& \leq \frac{C}{A} \lt( |\kappa| \|n\|_{L^\infty}^2 + |\mu| \|n\|_{L^\infty}\rt)\| n\|_{L_x^\infty L_y^2}^2 \| n \|_{L_x^2 L_y^\infty}^2. \alabel{eq:L^infty L^4''}
    	\end{align*}
        On one hand, we get
        \begin{align*}
            \| n (\cdot,y) \|_{L_x^\infty} &\leq \| n_k(y) \|_{L_k^1} \leq \| \langle k \rangle^m \langle \frac{1}{k} \rangle^\epsilon |k|^\frac16 n_k(y) \|_{L_k^2} \| |k|^{-\frac16} \langle k \rangle^{-m} \langle \frac{1}{k} \rangle^{-\epsilon}   \|_{L_k^2}\\
            &\lesssim \| \langle D_x \rangle^m \langle \frac{1}{D_x} \rangle^\epsilon n \|_{L_x^2}^\frac12 \| \langle D_x \rangle^m \langle \frac{1}{D_x} \rangle^\epsilon |D_x|^\frac13 n \|_{L_x^2}^\frac12,
        \end{align*}
        which implies 
        \begin{align*}
            \| n\|_{L_x^\infty L_y^2} \lesssim \| \langle D_x \rangle^m \langle \frac{1}{D_x} \rangle^\epsilon   n \|_{L^2}^\frac12 \| \langle D_x \rangle^m \langle \frac{1}{D_x} \rangle^\epsilon |D_x|^\frac13 n \|_{L^2}^\frac12.
        \end{align*}
        On the other hand, we infer that
        \begin{align*}
            \| n \|_{L_x^2 L_y^\infty} & =  \big\| \|n \|_{L_y^\infty} \big\|_{L_x^2 } \leq \big\| \|n\|_{L_y^2}^\frac12   \|\py n\|_{L_y^2}^\frac12 \big\|_{L_x^2 } \\
            &\leq \|n\|_{L^2}^\frac12 \|\py n\|_{L^2}^\frac12 \lesssim \| D_x \rangle^m \langle \frac{1}{D_x} \rangle^\epsilon n\|_{L^2}^\frac12 \| D_x \rangle^m \langle \frac{1}{D_x} \rangle^\epsilon \nabla n\|_{L^2}^\frac12.
        \end{align*}
        Hence, integrating \eqref{eq:L^infty L^4''} over $(0, t)$, we have
    	\begin{align*}
    		\|n\|_{L^\infty L^4}^4 &\leq \|n_{\mathrm{in}}\|_{L^4}^4 + \frac{C}{A}\lt( |\kappa| \|n\|_{L^\infty L^\infty}^2 + |\mu| \|n\|_{L^\infty L^\infty} \rt) \\
            &\qqquad \qqquad\times \| \langle D_x \rangle^m \langle \frac{1}{D_x} \rangle^\epsilon |D_x|^\frac13 n \|_{L^2L^2} \| \langle D_x \rangle^m \langle \frac{1}{D_x} \rangle^\epsilon \nabla n \|_{L^2L^2} \| \langle D_x \rangle^m \langle \frac{1}{D_x} \rangle^\epsilon n \|_{L^\infty L^2}^2\\
            &\leq \frac12\|n_{\mathrm{in}}\|_{L^\infty}^4 + \frac12 \|n \|_{L^\infty L^2}^4 + \frac{C}{A^\frac13} \lt( |\kappa| \|n\|_{L^\infty L^\infty}^2 + |\mu| \|n\|_{L^\infty L^\infty} \rt)   \| \langle D_x \rangle^m \langle \frac{1}{D_x} \rangle^\epsilon n \|_{X_a}^4. \alabel{eq:L^infty L^4 2}
    	\end{align*}

    \underline{\bf Step II: $L^\infty L^p$ estimate of n.} 
    For $ p=2^{j} $ with $ j=3,4...$, multiplying $\eqref{eq:main}_1$ by $2pn^{2p-1}$, integrating by parts over $\mathbb{R}^2$, and using Lemma \ref{lem:Elliptic estimate}, one deduces
		\begin{equation}\label{eq:n^p}
			\begin{aligned}
				&\dfrac{d}{dt}\|n^p\|_{L^2}^2+\dfrac{2(2p-1)}{pA}\|\nabla n^p\|_{L^2}^2\\
				=&\frac{2\mu(2p-1)}{A}\int_{ \mathbb{R}^2}n^p\nabla c\cdot\nabla n^pdxdy  +  \frac{2\kappa p(2p-1)}{(2p+1)A}\int_{ \mathbb{R}^2}\nabla c\cdot\nabla n^{2p+1}dxdy\\
				:=&I_{nl1}+I_{nl2}.
			\end{aligned}
		\end{equation}
		For $I_{nl1}$, it holds
        \begin{equation}\label{eq:Inl1}
			\begin{aligned}
				I_{nl1}&\leq\dfrac{2|\mu|(2p-1)}{A}\|\nabla n^p\|_{L^2}\|n^p\nabla c\|_{L^2}\\
				&\leq\dfrac{2p-1}{4pA}\|\nabla n^p\|_{L^2}^2+\dfrac{2\mu^2 p(2p-1)}{A}\|n^p\nabla c\|_{L^2}^2\\
				&\leq\dfrac{2p-1}{4pA}\|\nabla n^p\|_{L^2}^2+\dfrac{C\mu^2 p(2p-1)}{A}\|n^p\|_{L^2}\|\nabla n^{p}\|_{L^{2}}\|\nabla c\|_{L^4}^2\\
				&\leq\dfrac{2p-1}{2pA}\|\nabla n^p\|_{L^2}^2+\dfrac{C \mu^4 p^3(2p-1)}{A}\|n^p\|_{L^2}^2\|\nabla c\|_{L^4}^4.
			\end{aligned}
		\end{equation}
        For $I_{nl2}$, it follows from $\eqref{eq:main}_2$ that
		\begin{equation}\label{eq:Inl2}
			\begin{aligned}
				I_{nl2}&=-\frac{2\kappa p(2p-1)}{(2p+1)A}\int_{\mathbb{R}^2}n^{2p+1}\Delta cdxdy\\ 
                &=\frac{2\kappa p(2p-1)}{(2p+1)A}\int_{ \mathbb{R}^2}n^{2p+2}dxdy-\frac{2\kappa p(2p-1)\alpha}{(2p+1)A}\int_{ \mathbb{R}^2}n^{2p+1}cdxdy\\
				&\leq C\frac{2|\kappa| p(2p-1)}{(2p+1)A}\left(\|n^p\|_{L^2}^2\|\nabla n^p\|_{L^2}^{\frac2p}+\|n\|_{L^2}\| n^p\|_{L^2}\|\nabla n^p\|_{L^2}^{\frac{p+1}{p}}\right)\\
                &\leq\frac{2p-1}{2pA}\|\nabla n^{p}\|_{L^{2}}^{2}+\frac{C |\kappa|^{\frac{p}{p-1}}(2p-1)(p-1)}{pA}\left(\frac{p}{2p+1} \right)^{\frac{p}{p-1}}\|n^{p}\|_{L^{2}}^{\frac{2p}{p-1}}\\
                &\quad +\frac{C|\kappa|^{\frac{2p}{p-1}}(p-1)(2p-1)}{p(p+1)A}\left[\frac{p(p+1)}{2p+1} \right]^{\frac{2p}{p-1}}\|n\|_{L^{2}}^{\frac{2p}{p-1}}\|n^{p}\|_{L^{2}}^{\frac{2p}{p-1}}.
			\end{aligned}
		\end{equation}

		Adding \eqref{eq:n^p}--\eqref{eq:Inl2} together, we obtain
		\begin{equation*} 
			\begin{aligned}
				\frac{d}{dt}\|n^{p}\|_{L^{2}}^{2}\leq&-\frac{1}{A}\|\nabla n^{p}\|_{L^{2}}^{2}+\frac{C|\kappa|^{\frac{p}{p-1}} p}{A}\|n^{p}\|_{L^{2}}^{\frac{2p}{p-1}}+\frac{C|\kappa|^{\frac{2p}{p-1}}p^{3}}{A}\|n\|_{L^{2}}^{\frac{2p}{p-1}}\|n^{p}\|_{L^{2}}^{\frac{2p}{p-1}}\\
				&+\dfrac{C\mu^4 p^4}{A}\|n^p\|_{L^2}^2\|\nabla c\|_{ L^4}^4,
			\end{aligned}
		\end{equation*}
		which, along with the Nash inequality
		$$
		-\|\nabla n^p\|_{L^2}^{2}\leq-C\dfrac{\|n^p\|_{L^2}^4}{\|n^p\|_{L^1}^{2}},
		$$
		yields that
			\begin{align*}
				\frac{d}{dt}\|n^{p}\|_{L^{2}}^{2}&\leq -\frac{C}{A}\frac{\|n^{p}\|_{L^{2}}^{4}}{\|n^{p}\|_{L^{1}}^{2}}+\frac{C|\kappa|^{\frac{p}{p-1}}p}{A}\|n^{p}\|_{L^{2}}^{\frac{2p}{p-1}}+\frac{C|\kappa|^{\frac{2p}{p-1}}p^{3}}{A}\|n\|_{L^{2}}^{\frac{2p}{p-1}}\|n^{p}\|_{L^{2}}^{\frac{2p}{p-1}}
				+\frac{C\mu^4 p^4}{A}\|n^p\|_{L^2}^2\|\nabla c\|_{ L^4}^4\\
				&\leq-\frac{C\|n^p\|_{L^2}^2}{A\|n^{p}\|_{L^{1}}^{2}}\big(\|n^{p}\|_{L^{2}}^2-|\kappa|^{\frac{p}{p-1}}p\|n^{p}\|_{L^{1}}^{2}\|n^p\|_{L^2}^{\frac{2}{p-1}}   -     |\kappa|^{\frac{2p}{p-1}} p^{3}  \|n^p\|_{L^1}^2  \|n\|_{L^{2}}^{\frac{2p}{p-1}}\|n^p\|_{L^2}^{\frac{2}{p-1}}\\
				&\qqquad \qqquad\quad  - \mu^4 p^4\|n^p\|_{L^1}^2\|\nabla c\|_{L^4}^4\big)\\
				&\leq-\frac{C\|n^p\|_{L^2}^2}{A\|n^{p}\|_{L^{1}}^{2}}\big(\|n^p\|_{L^2}^2  - |\kappa|^{\frac{p}{p-2}} p^{\frac{p-1}{p-2}}\|n^p\|_{L^1}^{\frac{2(p-1)}{p-2}} -  |\kappa|^{\frac{2p}{p-2}} p^{\frac{3(p-1)}{p-2}}\|n^p\|_{L^1}^{\frac{2(p-1)}{p-2}}  \|n\|_{L^2}^{\frac{2p}{p-2}}\\
				&\qqquad \qqquad\quad-\mu^4 p^4\|n^p\|_{L^1}^2\|\nabla c\|_{ L^4}^4\big).
			\end{align*}
		Hence, proceeding by contradiction, we conclude that
		\begin{align*}
				&\quad \sup_{s\in [0, t]} \|n^p\|_{L^2}^2 \\
                &\leq C\bigg(\|n^{p}_{\mathrm{in}}\|_{L^{2}}^{2}  +  |\kappa|^{\frac{p}{p-2}}  p^{\frac{p-1}{p-2}}\|n^{p}\|_{L^{\infty}L^{1}}^{\frac{2(p-1)}{p-2}} \\
                &\qqquad+  |\kappa|^{\frac{2p}{p-2}} p^{\frac{3(p-1)}{p-2}}\|n\|_{L^{\infty}L^{2}}^{\frac{2p}{p-2}}\|n^{p}\|_{L^{\infty}L^{1}}^{\frac{2(p-1)}{p-2}}+ \mu^4 p^4\|n^p\|_{L^{\infty}L^1}^2\|\nabla c\|_{L^\infty L^4}^4 + 1\bigg)\\
			&= C\bigg(\|n_{\rm in}\|_{L^{2p}}^{2p}  +  |\kappa|^{\frac{p}{p-2}}  p^{\frac{p-1}{p-2}}\|n\|_{L^{\infty}L^{p}}^{\frac{2p(p-1)}{p-2}} \\
            &\qqquad +  |\kappa|^{\frac{2p}{p-2}}  p^{\frac{3(p-1)}{p-2}}\|{n}\|_{L^{\infty}L^{2}} ^{\frac{2p}{p-2}}\|n\|_{L^{\infty}L^{p}}^{\frac{2p(p-1)}{p-2}}+  \mu^4 p^{4}\|\nabla c\|_{L^{\infty}L^{4}}^{4}\|n\|_{L^{\infty}L^{p}}^{2p} + 1\bigg)\\
				&\leq C \lt(|\kappa|^3\|{n}\|_{L^{\infty}L^{2}}^3+\mu^4 \|\nabla c\|_{L^{\infty}L^{4}}^4  + \kappa^2 + 1\rt)\\
                &\qquad\times \left(\Vert{n_{\mathrm{in}}}\Vert_{L^{\infty}}^{2p-2}+p^{4}\|n\|_{L^{\infty}L^{p}}^{\frac{2p(p-1)}{p-2}}+\frac{p-2}{p-1}p^{\frac{4(p-1)}{p-2}}\|n\|_{L^{\infty}L^{p}}^{\frac{2p(p-1)}{p-2}} + 1\right)\\
				&\leq C\lt(|\kappa|^3\|{n}\|_{L^{\infty}L^{2}}^3+\mu^4 \|\nabla c\|_{L^{\infty}L^{4}}^4  + \kappa^2 + 1\rt)\max\left\{\left(1+\Vert{n_{\mathrm{in}}}\Vert_{L^{\infty}}\right)^{2p-2},p^{8}\|n\|_{L^{\infty}L^{p}}^{\frac{2p(p-1)}{p-2}} \right\}, \alabel{eq:claim}
			\end{align*}
        where we used that $\frac{p}{p-2} \leq 2$ and $\frac{2p}{p-2}\leq3$.

    \underline{\bf Step III: $L^\infty L^\infty$ estimate of n.} Denote $B:=C\lt(|\kappa|^3\|{n}\|_{L^{\infty}L^{2}}^3+\mu^4 \|\nabla c\|_{L^{\infty}L^{4}}^4  + \kappa^2 + 1\rt)$ and \\$D:=1+\Vert{n_{\mathrm{in}}}\Vert_{L^{\infty}}$. 
    Recalling $ p=2^{j}$ with $j\geq 2$, by iteration we infer from \eqref{eq:claim} that
		\begin{align*}
			 &\quad\sup_{s\in [0, t]}\int_{ \mathbb{R}^2 }|n|^{2^{j+1}}dxdy \\
			&\leq B \cdot 2^{8j}\left(\sup_{s\in [0, t]}\int_{ \mathbb{R}^2  }|n|^{2^{j}}dxdy \right)^{\frac{2^{j+1}-2}{2^{j}-2} }  \\
			&\leq B \cdot 2^{8j} \left(  B \cdot 2^{8(j-1)}\left(\sup_{s\in [0, t]}\int_{ \mathbb{R}^2  }|n|^{2^{j-1}}dxdy \right)^{\frac{2^{j}-2}{2^{j-1}-2} } \right)^{\frac{2^{j+1}-2}{2^{j}-2} } \\
            &= B^{1 + \frac{2^{j+1}-2}{2^{j}-2}} \cdot 2^{8j + 8(j-1)\frac{2^{j+1}-2}{2^{j}-2}} \left(\sup_{s\in [0, t]}\int_{ \mathbb{R}^2  }|n|^{2^{j-1}}dxdy \right)^{\frac{2^{j+1}-2}{2^{j}-2}\cdot\frac{2^{j}-2}{2^{j-1}-2}} \\
            &\leq B^{1 + \frac{2^{j+1}-2}{2^{j}-2}} \cdot 2^{8j + 8(j-1)\frac{2^{j+1}-2}{2^{j}-2}} \left(B \cdot 2^{8(j-2)}\left(\sup_{s\in [0, t]}\int_{ \mathbb{R}^2  }|n|^{2^{j-2}}dxdy \right)^{\frac{2^{j-1}-2}{2^{j-2}-2} } \right)^{\frac{2^{j+1}-2}{2^{j}-2}\cdot\frac{2^{j}-2}{2^{j-1}-2}}\\
            &= B^{1 + \frac{2^{j+1}-2}{2^{j}-2} + \frac{2^{j+1}-2}{2^{j}-2}\cdot\frac{2^{j}-2}{2^{j-1}-2} } \cdot 2^{8j + 8(j-1)\frac{2^{j+1}-2}{2^{j}-2}  + 8(j-2) {\frac{2^{j+1}-2}{2^{j}-2}\cdot\frac{2^{j}-2}{2^{j-1}-2}}} \\
            &\qquad \times \left(\sup_{s\in [0, t]}\int_{ \mathbb{R}^2  }|n|^{2^{j-2}}dxdy \right)^{ {\frac{2^{j+1}-2}{2^{j}-2}\cdot\frac{2^{j}-2}{2^{j-1}-2}}\cdot \frac{2^{j-1}-2}{2^{j-2}-2} }\\
            &\quad...\\
            &\leq B^{ (2^{j+1} -2)\Sigma_{i=3}^{j+1} \frac{1}{2^i -2 }}  2^{8 (2^{j+1} -2)\Sigma_{i=3}^{j+1} \frac{i-1}{2^i -2}}  \left(\sup_{s\in [0, t]}\int_{ \mathbb{R}^2  }|n|^{2^{2}}dxdy \right)^{\frac{2^{j+1}-2}{2^{2}-2}}, \alabel{Moser1}
		\end{align*}
       where we assumed the maximum was attained by the iterative term. Otherwise, let $j\geq j_0(j)\geq 3$ be the first step where the maximum is determined by the first term, i.e., 
        \begin{align*}
            &\quad\sup_{s\in [0, t]}\int_{ \mathbb{R}^2 }|n|^{2^{j+1}}dxdy \\
			&\leq B^{ (2^{j+1} -2)\Sigma_{i=j_0+1}^{j+1} \frac{1}{2^i -2 }}  2^{8 (2^{j+1} -2)\Sigma_{i=j_0+1}^{j+1} \frac{i-1}{2^i -2}}  \left(\sup_{s\in [0, t]}\int_{ \mathbb{R}^2  }|n|^{2^{j_0}}dxdy \right)^{\frac{2^{j+1}-2}{2^{j_0}-2}} \\
            &\leq B^{ (2^{j+1} -2)\Sigma_{i=j_0+1}^{j+1} \frac{1}{2^i -2 }}  2^{8 (2^{j+1} -2)\Sigma_{i=j_0+1}^{j+1} \frac{i-1}{2^i -2}}  \left( B D^{2\cdot2^{j_0 -1}  -2 } \right)^{\frac{2^{j+1}-2}{2^{j_0}-2}} \\
            &\leq B^{ (2^{j+1} -2)\Sigma_{i=3}^{j+1} \frac{1}{2^i -2 }}  2^{8 (2^{j+1} -2)\Sigma_{i=3}^{j+1} \frac{i-1}{2^i -2}}  D^{2^{j+1}-2}.  \alabel{Moser2}
        \end{align*}
        In this case, one can find that the bound is already established, and the subsequent iteration steps are unnecessary.
        
        Direct calculations show
       \begin{align*}
        \sum_{i=3}^{\infty} \frac{1}{2^i -2 } \leq \sum_{i=1}^{\infty} \frac{1}{2^i  } \leq 1, \quad
        \sum_{i=3}^{\infty} \frac{i-1}{2^i -2 } \leq \sum_{i=1}^{\infty} \frac{i}{2^i  } \leq 2.
       \end{align*}
       Then, raising both sides of \eqref{Moser1} and \eqref{Moser2} to the power of $\frac{1}{2^{j+1}}$, and letting  $j\rightarrow+\infty$, we obtain
       \begin{align*}
           \| n \|_{L^\infty L^\infty} &\leq 2^{16} B \max\{\| n\|_{L^\infty L^4}^2 , D \}\\
           &\leq C\lt(|\kappa|^3\|{n}\|_{L^{\infty}L^{2}}^3+\mu^4 \|\nabla c\|_{L^{\infty}L^{4}}^4  + \kappa^2 + 1\rt)\lt( \| n\|_{L^\infty L^4}^2 +  \Vert{n_{\mathrm{in}}}\Vert_{L^{\infty}} +1\rt),
       \end{align*}
       which, combined with \eqref{eq:L^infty L^4} and \eqref{eq:L^infty L^4 2}, completes the proof.
    \end{proof}

\subsection{The energy estimate of $\omega$.} 
    In this subsection, we give the energy estimate of $\omega$.
    
    \begin{Lem} \label{lem:est of omega}
    	For any $0<\epsilon<1/2<m$, $0<a<\frac{1}{16(1+2 \pi)}$ and $0\leq t \leq T$, it holds that 
    	\begin{equation*}\begin{aligned}
    			\|  \langle D_x\rangle^m  \langle \frac{1}{D_x}\rangle^\epsilon   \omega  \|_{X_a}^2 &\lesssim   \|  \langle D_x\rangle^m  \langle \frac{1}{D_x}\rangle^\epsilon   \omega_{\mathrm{in}} \|_{L^2}^2 + \frac{1}{A^\frac12}  \|  \langle D_x\rangle^m  \langle \frac{1}{D_x}\rangle^\epsilon   \omega \|_{X_a}^3  \\
    			&\qquad + \frac{1}{A^\frac23} \|  \langle D_x\rangle^m  \langle \frac{1}{D_x}\rangle^\epsilon   \omega \|_{X_a}   \|  \langle D_x\rangle^{m+\frac13}  \langle \frac{1}{D_x}\rangle^\epsilon   n \|_{X_a}.
    	\end{aligned}\end{equation*}
    \end{Lem}
    \begin{proof}
    	Applying Proposition \ref{lem:est of f} to $\eqref{eq:main}_3$, we have
    	\begin{equation}\begin{aligned} \label{eq:est of omega1}
    			\|  \langle D_x\rangle^m  \langle \frac{1}{D_x}\rangle^\epsilon   \omega  \|_{X_a}^2 &\lesssim   \|  \langle D_x\rangle^m  \langle \frac{1}{D_x}\rangle^\epsilon   \omega_{\mathrm{in}} \|_{L^2}^2 + \frac{1}{A^\frac12}  \|  \langle D_x\rangle^m  \langle \frac{1}{D_x}\rangle^\epsilon   \omega \|_{X_a}^3  \\
    			&\qquad + \frac1A \int_{0}^{t}  \left|\Re\left(\px n \left\lvert\, \mathcal{M} e^{2 a A^{-\frac{1}{3}}\left|D_x\right|^{\frac{2}{3} }t }\langle D_x\rangle^{2 m}\langle\frac{1}{D_x}\rangle^{2 \epsilon} \omega\right.\right)\right|  dt.
    	\end{aligned}\end{equation}
    	Moreover, inspired by the handling of the buoyancy term $\partial_x \theta$ in \cite{CWY2025a}, we infer that
    	\begin{equation}\begin{aligned} \label{eq:px n omega}
    			&\frac1A \int_{0}^{t}  \left|\Re\left(\px n \left\lvert\, \mathcal{M} e^{2 a A^{-\frac{1}{3}}\left|D_x\right|^{\frac{2}{3} }t }\langle D_x\rangle^{2 m}\langle\frac{1}{D_x}\rangle^{2 \epsilon} \omega\right.\right)\right|  dt \\
    			\lesssim & \frac1A \|e^{a A^{-\frac{1}{3}}\left|D_x\right|^{\frac{2}{3} } t }  \langle D_x\rangle^m\langle\frac{1}{D_x}\rangle^\epsilon |D_x|^\frac13 \omega\|_{L^2 L^2}  \|e^{a A^{-\frac{1}{3}}\left|D_x\right|^{\frac{2}{3} } t }  \langle D_x\rangle^m\langle\frac{1}{D_x}\rangle^\epsilon |D_x|^\frac23 n\|_{L^2 L^2} \\
    			\lesssim & \frac{1}{A^\frac23} \|  \langle D_x\rangle^m  \langle \frac{1}{D_x}\rangle^\epsilon   \omega \|_{X_a}   \|  \langle D_x\rangle^{m+\frac13}  \langle \frac{1}{D_x}\rangle^\epsilon   n \|_{X_a}.
    	\end{aligned}\end{equation}
    	Combining \eqref{eq:px n omega} with \eqref{eq:est of omega1}, the proof is complete.
    \end{proof}
    
    \subsection{Closing the energy}
    With the a priori estimates for $n$ and $\omega$ at hand, we can let $A$ sufficiently large so that the energy can be closed as follows.
    
    \begin{proof}[Proof of Proposition \ref{main prop}]
    	By \eqref{eq:bootstap}, Lemma \ref{lem:est of n} and Lemma \ref{lem:est of omega}, we get for $0<t\leq T$ that
    	\begin{align*} 
    			E(t)^2 &\leq C \|  \langle D_x\rangle^m  \langle \frac{1}{D_x}\rangle^\epsilon   \omega_{\mathrm{in}} \|_{L^2}^2 + \frac{C}{A^\frac12}  \|  \langle D_x\rangle^m  \langle \frac{1}{D_x}\rangle^\epsilon   \omega \|_{X_a}^3   \\
    			&\quad + \frac{C}{A^\frac23} \|  \langle D_x\rangle^m  \langle \frac{1}{D_x}\rangle^\epsilon   \omega \|_{X_a}   \|  \langle D_x\rangle^{m+\frac13}  \langle \frac{1}{D_x}\rangle^\epsilon   n \|_{X_a}  \\
    			&\quad +C \|  \langle D_x\rangle^{m +\frac13}  \langle \frac{1}{D_x}\rangle^\epsilon   n_{\mathrm{in}} \|_{L^2}^2  + \|  \langle D_x\rangle^{m+\frac13}  \langle \frac{1}{D_x}\rangle^\epsilon   n \|_{X_a}^2  \\
    			&\quad \quad \times \lt( \frac{C}{A^\frac12}  \|  \langle D_x\rangle^m  \langle \frac{1}{D_x}\rangle^\epsilon   \omega \|_{X_a}  +  \frac{C|\mu|}{A^\frac13} \|  \langle D_x\rangle^{m+\frac13}   \langle \frac{1}{D_x}\rangle^\epsilon   n \|_{X_a}  +  \frac{C|\kappa|}{A^\frac16} \|  \langle D_x\rangle^{m+\frac13}   \langle \frac{1}{D_x}\rangle^\epsilon   n \|_{X_a}^2  \rt)\\
    			&\leq C\lt(\|  \langle D_x\rangle^m  \langle \frac{1}{D_x}\rangle^\epsilon   \omega_{\mathrm{in}} \|_{L^2}^2  +  \|  \langle D_x\rangle^{m +\frac13}  \langle \frac{1}{D_x}\rangle^\epsilon   n_{\mathrm{in}} \|_{L^2}^2  \rt) \\
    			&\quad +  C\lt[\frac{(2K)^3}{A^\frac12}  +  \frac{(2K)^2}{A^\frac23}  + \frac{(2K)^3}{A^\frac12} + \frac{|\mu|(2K)^3}{A^\frac13}   + \frac{|\kappa|(2K)^4}{A^\frac16} \rt].\alabel{eq:temp1}
    	\end{align*}
    	Choose  $\bar{A}_1 = \bar{A}_1(K, \mu, \kappa):= C \lt(|\kappa|^6K^{24} + |\mu|^3K^9 + K^6\rt) $ large enough such that if $A> \bar{A}_1$, it holds that
    	\begin{equation}\begin{aligned} \label{eq:bar A1}
    			C\lt[\frac{(2K)^3}{A^\frac12}  +  \frac{(2K)^2}{A^\frac23}  + \frac{(2K)^3}{A^\frac12} + \frac{|\mu|(2K)^3}{A^\frac13} + \frac{|\kappa|(2K)^4}{A^\frac16} \rt] \leq 1.
    	\end{aligned}\end{equation}
    	Now we denote 
        $$
        K:= \sqrt{C\lt(\|  \langle D_x\rangle^m  \langle \frac{1}{D_x}\rangle^\epsilon   \omega_{\mathrm{in}} \|_{L^2}^2  +  \|  \langle D_x\rangle^{m +\frac13}  \langle \frac{1}{D_x}\rangle^\epsilon   n_{\mathrm{in}} \|_{L^2}^2  \rt) + 1}.
        $$
        Then, by Lemma \ref{lem:n L infty L infty}, Lemma \ref{lem:Elliptic estimate}, \eqref{eq:temp1} and \eqref{eq:bar A1}, we have for $\alpha>0$ that
    	\begin{align*}
           E_{\infty}(t)
           &\leq C\lt(|\kappa|^3\|{n}\|_{L^{\infty}L^{2}}^3+\mu^4 \|\nabla c\|_{L^{\infty}L^{4}}^4  + \kappa^2 + 1\rt)\\
           &\qquad\times \lt[ \|n_{\mathrm{in}}\|_{L^\infty}^2 +  \|n \|_{L^\infty L^2}^2 + \frac{1}{A^\frac13} \lt( |\kappa| \|n\|_{L^\infty L^\infty}^2 + |\mu| \|n\|_{L^\infty L^\infty} \rt)  \| \langle D_x \rangle^m \langle \frac{1}{D_x} \rangle^\epsilon n \|_{X_a}^2  +1\rt]\\
           &\leq C\lt(|\kappa|^3\|{n}\|_{X_a}^3+ \alpha^{-1}\mu^4 \|n\|_{X_a}^4   + \kappa^2 + 1\rt)\\
           &\qquad\times \lt[ \|n_{\mathrm{in}}\|_{L^\infty}^2 +  \|n \|_{X_a}^2 + \frac{1}{A^\frac13}  \lt( |\kappa| \|n\|_{L^\infty L^\infty}^2 + |\mu| \|n\|_{L^\infty L^\infty}\rt)    \| \langle D_x \rangle^m \langle \frac{1}{D_x} \rangle^\epsilon n \|_{X_a}^2  +1\rt]\\
           &\leq C\lt(|\kappa|^3K^3+ \mu^4 K^4 + \kappa^2 + 1\rt)  \lt( \|n_{\mathrm{in}}\|_{L^\infty}^2 +  K^2 + \frac{  |\kappa|K_\infty^2 + |\mu| K_\infty }{A^\frac13} K^2 +1\rt),
       \end{align*}
    and for $\alpha=0$ that
    \begin{align*}
           E_{\infty}(t)
           &\leq C\lt(|\kappa|^3\|{n}\|_{L^{\infty}L^{2}}^3+\mu^4 \|\nabla c\|_{L^{\infty}L^{4}}^4  + \kappa^2 + 1\rt)\\
           &\qquad\times \lt[ \|n_{\mathrm{in}}\|_{L^\infty}^2 +  \|n \|_{L^\infty L^2}^2 + \frac{1}{A^\frac16} \lt( |\kappa| \|n\|_{L^\infty L^\infty}^2 + |\mu| \|n\|_{L^\infty L^\infty} \rt)^{\frac12}  \| \langle D_x \rangle^m \langle \frac{1}{D_x} \rangle^\epsilon n \|_{X_a}^2  +1\rt]\\
           &\leq C\lt(|\kappa|^3\|{n}\|_{X_a}^3+ \mu^4 \|n\|_{X_a}^4+\mu^4 M^4   + \kappa^2 + 1\rt)\\
           &\qquad\times \lt[ \|n_{\mathrm{in}}\|_{L^\infty}^2 +  \|n \|_{X_a}^2 + \frac{1}{A^\frac16}  \lt( |\kappa| \|n\|_{L^\infty L^\infty}^2 + |\mu| \|n\|_{L^\infty L^\infty}\rt) ^{\frac12}  \| \langle D_x \rangle^m \langle \frac{1}{D_x} \rangle^\epsilon n \|_{X_a}^2  +1\rt]\\
           &\leq C\lt(|\kappa|^3K^3+ \mu^4 K^4 + \mu^4 M^4 + \kappa^2 + 1\rt)  \lt[\|n_{\mathrm{in}}\|_{L^\infty}^2 +  K^2 + \lt(\frac{  |\kappa|K_\infty^2 + |\mu| K_\infty }{A^\frac13}\rt)^{\frac12} K^2 +1\rt].
       \end{align*}
    	Choose  $\bar{A}_2 = \bar{A}_2(K_\infty, \mu, \kappa):=  C\lt(|\kappa|^3 K_\infty^{6} + |\mu|^3 K_\infty^3 \rt)$ large enough such that if $A> \bar{A}_2$, it holds that
    	\begin{equation*}\begin{aligned}
    			\frac{ |\kappa|K_\infty^2 + |\mu| K_\infty }{A^\frac13} \leq 1.
    	\end{aligned}\end{equation*}
        Then, we define
        \begin{align*} 
    		K_\infty:= \left\{\begin{array}{l}
        C \lt(|\kappa|^3K^3+ \mu^4 K^4 + \kappa^2 + 1\rt)  \lt( \|n_{\mathrm{in}}\|_{L^\infty}^2 +  K^2  +1\rt) ,\, \text{if} \,\, \alpha>0 ;\\
        C \lt(|\kappa|^3K^3+ \mu^4 K^4 + \mu^4 M^4 + \kappa^2 + 1\rt)  \lt( \|n_{\mathrm{in}}\|_{L^\infty}^2 +  K^2  +1\rt) ,\, \text{if} \,\, \alpha=0.
        \end{array}\right.
        \end{align*}
        Let $\bar{A} = \max\lt\{\bar{A}_1, \bar{A}_2\rt\},$ and then the proof is complete.
    \end{proof}

    \section*{Declarations}
    \begin{itemize}
    \item \textbf{Acknowledgements} W. Wang was supported by National Key R\&D Program of China (No. 2023YFA1009200) and NSFC under grant 12471219.
    \item \textbf{Conflict of interest} The authors declare that they have no conflict of interest.
    \item \textbf{Data Availability} Data sharing is not applicable to this article as no datasets were generated or analyzed during the current study.
    \end{itemize}

    
    \bibliographystyle{amsalpha}
    \bibliography{RRNNrefs}
\end{document}